\documentclass[11pt]{amsart} 
\usepackage[lmargin=1in,rmargin=1in,tmargin=1in,bmargin=1in]{geometry}
\usepackage[ps,all,arc,rotate]{xy}
\usepackage{graphicx, float, epstopdf}
\usepackage{hyperref, color}
\usepackage{centernot}
\usepackage{fancyhdr}
\usepackage{amsfonts,amssymb,amsmath,amsthm}
\usepackage[usenames,dvipsnames]{xcolor}
\usepackage{datetime}
\numberwithin{equation}{section}
\numberwithin{figure}{section}
\allowdisplaybreaks[4]                        
 
\newtheorem{lemma}{Lemma}[section]
\newtheorem{theorem}{Theorem}[section]

\theoremstyle{definition}

\newtheorem{remark}{Remark}[section]
\newcommand{\R}{\mathbb{R}}

\newcommand{\real}{\operatorname{Re}}

\begin{document}
\title[On a mollifier of the perturbed Riemann zeta-function]{On a mollifier of the perturbed Riemann zeta-function}
\author{Patrick K\"{u}hn}
\address{Institut f\"{u}r Mathematik, Universit\"{a}t Z\"{u}rich, Winterthurerstrasse 190 \\ CH-8057 Z\"{u}rich, Switzerland}
\email{patrick.kuehn@math.uzh.ch} 
\author{Nicolas Robles}
\address{Department of Mathematics, University of Illinois, 1409 West Green Street \\ Urbana, IL 61801, United States}
\email{nirobles@illinois.edu} 
\author{Dirk Zeindler}
\address{Department of Mathematics and Statistics, Lancaster University, Fylde College, Bailrigg, Lancaster LA1 4YF, United Kingdom}
\email{d.zeindler@lancaster.ac.uk}
\maketitle
\begin{abstract}
The mollification $\zeta(s) + \zeta'(s)$ put forward by Feng is computed by analytic methods coming from the techniques of the ratios conjectures of $L$-functions. The current situation regarding the percentage of non-trivial zeros of the Riemann zeta-function on the critical line is then clarified. 
\end{abstract}
\section{Introduction}
\subsection{Statement of the results}
\noindent The Riemann zeta-function $\zeta(s)$ is defined by the Dirichlet series $\zeta(s)=\sum_{n=1}^{\infty}n^{-s}$ for $s=\sigma+it$, $\sigma>1$ and $t \in \R$. The functional equation of $\zeta(s)$ is given by
\[
\xi(s) = \xi(1-s),
\]
where
\[
\xi(s) = H(s) \zeta(s) \quad \textnormal{and} \quad H(s) = \frac{1}{2}s(s-1)\pi^{-s/2}\Gamma \bigg( \frac{s}{2} \bigg).
\]
This allows us to perform a meromorphic continuation to the whole complex plane except at $s=1$ where $\zeta(s)$ has a simple pole with residue equal to $1$. The connection with number theory comes from the Euler product
\[
\zeta(s) = \prod_{p} (1 - p^{-s})^{-1},
\]
for $\real(s)>1$, and where the product is taken over all the primes $p$. It is well-known from Riemann and from von Mangoldt that the non-trivial zeros $\rho = \beta + i \gamma$ of $\zeta(s)$ are located inside the critical strip $0 < \beta <1$. Moreover, if $N(T)$ denotes the number of such zeros up to height $0 \le \gamma <T$ then
\[
N(T) = \frac{T}{2 \pi} \bigg(\log \frac{T}{2\pi} -1 \bigg) + \frac{7}{8} + S(T) + O \bigg( \frac{1}{T} \bigg),
\] 
where
\[
S(T) = \frac{1}{\pi} \arg \zeta \bigg( \frac{1}{2} +it \bigg) \ll \log T,
\]
as $T \to \infty$, see e.g. \cite{montvau,titchmarsh} for properties of $\zeta(s)$. To state the results, we let $N_0(T)$ denote the number of non-trivial zeros up to height $T>0$ such that $\beta=1/2$. Similarly, let $N_0^*(T)$ denote the number such zeros which are also simple. We then define
\[
\kappa  = \mathop {\lim \inf }\limits_{T \to \infty } \frac{{{N_0}(T)}}{{N(T)}} \quad \textnormal{and} \quad \kappa^*  = \mathop {\lim \inf }\limits_{T \to \infty } \frac{{{N_0^*}(T)}}{{N(T)}}.
\]
The history behind the value of $\kappa$ can be found in \cite{bcy,feng,rrz01}. The main breakthroughs were as follows. In 1942, Selberg \cite{selberg} established that $0 < \kappa \le 1$. Levinson later showed in 1974 that $\kappa \ge .3474$. This was improved by Conrey to $\kappa \ge .4088$ in 1989 and later refined by Bui, Conrey and Young \cite{bcy} to $\kappa \ge .4105$, and shortly afterward by Feng \cite{feng} to $\kappa \ge .4127$. It should be noted that both results are improvements of $\kappa \ge .4088$ and are independent of each other.\\

The second author, Roy and Zaharescu \cite{rrz01} as well as Bui \cite{bui} brought up a point regarding the strength of Feng's result. In \cite{rrz01}, it was explained that $\kappa \ge .4107$, unconditionally, using Feng's mollifier. However, the computation of the mixed terms of the mollifiers of Conrey and of Feng was not carried through explicitly.\\

It should also be remarked that Bui \cite{bui} suggests that the bound obtained in this paper can be attained using the twisted second moment of the Riemann zeta-function due to Balasubramanien \textsl{et. al.} \cite{bchb} and that he also suggests an alternative argument that could lead to the bound $\kappa > .41098$.\\

In this paper, we close this gap and we explain Feng's brilliant choice in the context of the powerful technology developed in \cite{bcy,youngsimple}. These ideas come from the ratios conjectures of $L$-functions due to Conrey, Farmer and Zirnbaeuer \cite{conreyfarmerzirnbauer} as well as to Conrey and Snaith \cite{conreysnaithratios}. It should be noted that Feng's methodology to obtain the main terms of his theorem consisted on an ingenious combination of elementary methods, namely induction and Mertens' formula, applied to Conrey's result \cite{conrey89}. On the other hand, this choice of methods blurred a bit the length the mollifier was allowed to take. Other than choosing the same mollifier, our computations do not overlap and the methods are quite different.\\

Lastly, the closing of this gap will clarify the situation of the percentage of non-trivial zeros on the critical line when one attaches Feng's second-piece mollifier to Conrey's. 
\subsection{Choice of mollifiers}
Let $Q(x)$ be a real polynomial satisfying $Q(0)=1$, $Q(x)+Q(1-x)=\operatorname{constant}$, and define
\begin{align} \label{defV}
V(s) = Q \bigg( -\frac{1}{L}\frac{d}{ds} \bigg) \zeta(s),
\end{align}
where for large $T$,
\[
L = \log T.
\]
If $\psi(s)$ is a mollifier, then it is well-known from the work of Levinson \cite{levinson} and of Conrey \cite{conrey89} that Littlewood's lemma \cite[$\mathsection$9.9]{titchmarsh} followed by the arithmetic and geometric mean inequalities yields
\begin{align} \label{kappaineq}
\kappa \ge 1 - \frac{1}{R} \log \bigg( \frac{1}{T} \int_{1}^{T} |V \psi(\sigma_0 + it)|^2 dt \bigg) +o(1),
\end{align}
where $\sigma_0=1/2-R/L$, and $R$ is a bounded positive real number to be chosen later. Following Feng \cite{feng}, we will choose a mollifier of the form
\[
\psi(s) = \psi_1(s) + \psi_2(s),
\]
where $\psi_1$ is the mollifier considered by Conrey. Let $P_1(x) = \sum_j a_j x^j$ be a certain polynomial satisfying $P_1(0)=0$, $P_1(1)=1$, and let $y_1 = T^{\theta_1}$ where $0 < \theta_1 < 4/7$. We adopt the notation
\[
P_1[n] = P_1 \bigg( \frac{\log(y_1/n)}{\log y_1} \bigg)
\]
for $1 \le n \le y_1$. By convention, we set $P_1[x]=0$ for $x \ge y_1$. Then $\psi_1(s)$ is given by
\begin{align} \label{ps1def}
\psi_1(s) = \sum_{h \le y_1} \frac{\mu(h)h^{\sigma_0 - 1/2}}{h^s} P_1[h],
\end{align}
where $\mu(n)$ is the M\"{o}bius function. For the second mollifier, we take
\begin{align} \label{ps1def}
\psi_2(s) = \sum\limits_{k \leqslant {y_2}} {\frac{{\mu (k){k^{{\sigma _0} - 1/2}}}}{{{k^s}}}} \sum\limits_{\ell  = 2}^K {\sum\limits_{{p_1} \cdots {p_\ell }|k} {\frac{{\log {p_1} \cdots \log {p_\ell }}}{{{{\log }^\ell }{y_2}}}} {P_\ell }[k]}. 
\end{align}
Here $K \ge 2$ is an integer of our choice and $p_1, \cdots, p_\ell$ are distinct primes. Also we need $P_\ell(0)=0$ for $\ell = 2, \cdots, K$. In this case $y_2 = T^{\theta_2}$ where $0< \theta_2 < 1/2$. 
\begin{remark}
It will become clear in the calculation of the crossterm integral between $\psi_1$ and $\psi_2$ that one needs $\theta_1 + \theta_2 < 1 - \varepsilon$. Therefore, if $\theta_1$ increases, then $\theta_2$ decreases unless some difficult work is done to push $\theta_2$ back to its original (or higher) value. See the comments between Theorem \ref{theorem1212} and Theorem \ref{theorem4737} for more details.
\end{remark}
The reason behind this choice is that Feng wishes to mollify not only $\zeta(s)$ but also $\frac{\zeta'(s)}{\log T}$, which is the second term coming from \eqref{defV}. This is accomplished by looking at
\begin{align} \label{expansion1overzeta}
\frac{1}{{\zeta (s) + \tfrac{{\zeta '(s)}}{{\log T}}}} = \frac{1}{{\zeta (s)}} - \frac{1}{{\log T}}\frac{{\zeta '}}{{{\zeta ^2}}}(s) + \frac{1}{{{{\log }^2}T}}\frac{{{{(\zeta ')}^2}}}{{{\zeta ^3}}}(s) - \frac{1}{{{{\log }^3}T}}\frac{{{{(\zeta ')}^3}}}{{{\zeta ^4}}}(s) +  \cdots .
\end{align}
When $k$ is a square-free positive integer, then one has
\[
(\mu  * {\Lambda ^{ * \ell }})(k) = ( - 1)^\ell  \mu(k)\sum\limits_{{p_1} \cdots {p_\ell }|k} {\log {p_1} \cdots \log {p_\ell }} ,
\]
where $f * g$ denotes the Dirichlet convolution of arithmetic functions $f$ and $g$. Here $\Lambda^{* \ell}$ stands for convolving the von Mangoldt function $\Lambda(n)$ with itself exactly $\ell$ times. If $k$ contains a square divisor, then, as remarked by Feng \cite{feng}, the coefficients $a_j$ resulting from \eqref{expansion1overzeta} contribute a lower order to the mean value integrals $I_{11}$, $I_{12}$ and $I_{22}$ related to $\kappa$ in \eqref{kappaineq} (see below for exact definitions of these $I$-integrals).
\subsection{Numerical evaluations}
We will prove the following.
\begin{theorem} 
\label{theorem1212}
We obtain with $\theta_1 = \theta_2 = 1/2 - \varepsilon$
\[
\kappa \ge .369927 \quad \textnormal{and} \quad \kappa^*  \ge .359991,
\]
unconditionally.
\end{theorem}
Using the work of Iwaniec and Deshouillers \cite{deshouillersIwaniec1,deshouillersIwaniec2}, Conrey \cite{conrey89} was able to push the size of the mollifier $\psi_1$ to $\theta_1 < 4/7 - \varepsilon$. In the light of Lemma~\ref{lemmaAFE} and \eqref{newpatherrorterms} below, we require $\theta_1 + \theta_2 < 1$ in our argumentation. The points brought up in \cite{bui} and \cite{rrz01} show that some difficult work is needed if one takes $\theta_1 + \theta_2 > 1 $. Theorem \ref{theorem1212} utilizes $\theta_1, \theta_2 < 1/2 - \varepsilon$. However, if we take $\theta_1 < 4/7 - \varepsilon$ and $\theta_2 < 3/7 - \varepsilon$, then we get
\begin{theorem} \label{theorem4737}
We obtain with $\theta_1 < 4/7 - \varepsilon$ and $\theta_2 < 3/7 - \varepsilon$
\[
\kappa \ge .410725 \quad \textnormal{and} \quad \kappa^*  \ge .403211,
\]
unconditionally.
\end{theorem}
It should therefore be stressed that Theorem \ref{theorem1212} is an improvement of the last theorem to ever use $\theta_1 = 1/2 - \varepsilon$, namely the first corollary of \cite{conrey83a}, where it was shown that $\kappa \ge .3658$.\\
The method sketched in \cite{bcy, rrz01} to deal with multiple piece mollifiers carries through and our main result is as follows.
\begin{theorem} \label{unsmoothed}
Suppose that $\theta_1 + \theta_{2} = 1 - \varepsilon$ with $\theta_1 < 4/7 $ and $\theta_{2} < 1/2 $ and $\varepsilon > 0$ small. Then 
\[
\frac{1}{T} \int_1^T |V \psi(\sigma_0 + it)|^2 dt = c(P_1,P_{\ell},Q,R,\theta_1,\theta_2) + o(1),
\]
where $c(P_1,P_{\ell},Q,R,\theta_1,\theta_2) = c_{11} + 2c_{12} + c_{22}$ and the $c_{ij}$ are given by \eqref{c11formula}, \eqref{c12formula} and \eqref{c22formula}.
\end{theorem}
We use \texttt{Mathematica} to numerically evaluate $c(P_1,P_{\ell},Q,R,1/2,1/2)$ with the following choices of parameters. With $K=3$, $R=1.3$,
\begin{align}
Q(x) &= .481936 + .632349(1-2x) - .144698(1-2x)^3 + .0304136 (1-2x)^5, \nonumber \\
P_1(x) &= x + .225339x(1-x) -1.01137 x(1-x)^2 + .174004x(1-x)^3 - .100235 x(1-x)^4, \nonumber \\
P_2(x) &= 1.05138 x + .284201x^2, \nonumber \\
P_3(x) &= .222032 x -.13254 x^2, \nonumber
\end{align}
we have $\kappa \ge .369927$. To get $\kappa^* \ge .359991$, we take $K=3$, $R=1.2$,
\begin{align}
Q(x) & = .476202 + .523798 (1-2x),\nonumber \\
P_1(x) &= x + .0531913x(1-x) - .594999 x(1-x)^2 -.00107597 x(1-x)^3 -.0761954 x(1-x)^4, \nonumber \\
P_2(x) &= .896567 x -.0297464 x^2, \nonumber \\
P_3(x) &= .0699271 x -.108964 x^2. \nonumber
\end{align}
We also use \texttt{Mathematica} to numerically evaluate $c(P_1,P_{\ell},Q,R,4/7,3/7)$ with the following choices of parameters. With $K=3$, $R=1.295$,
\begin{align}
Q(x) &= .492203 + .621972(1-2x) - .148163(1-2x)^3 + .033988 (1-2x)^5 , \nonumber \\
P_1(x) &= x + .229117x(1-x) - 2.932318x(1-x)^2 + 4.856163x(1-x)^3 - 2.390999x(1-x)^4 \nonumber \\
P_2(x) &= -.072644x + 1.559440x^2\nonumber \\
P_3(x) &= .701568x - .554403 x^2\nonumber
\end{align}
we have $\kappa \ge .410725$. To get $\kappa^* \ge .403211$, we take $K=3$, $R=1.109$,
\begin{align}
Q(x) &= .485034 + .514966(1-2x), \nonumber \\
P_1(x) &= x + .0486916x(1-x) -2.02526x(1-x)^2 +3.43611x(1-x)^3 -1.62355 x(1-x)^4, \nonumber \\
P_2(x) &= -.034431x + 1.09223x^2, \nonumber \\
P_3(x) &= .479296x -0.385868x^2. \nonumber
\end{align}
\indent An interesting question to ask is: what would have happened if Feng had published his mollifier before Conrey's increment of $\theta_1$ from $1/2$ to $4/7$. Since this has not been remarked before in the literature, we take the chance to answer it. If $\psi_1$ and $\psi_2$ are kept at $1/2 - \varepsilon$, then Feng's piece adds an additional $0.4127\%$ to Conrey's $36.58\%$ as shown in the table below.
\begin{table}[H]
\centering
\begin{tabular}{ccccc}
 $\theta_1$ & $\theta_2$ & \%  \\
 $1/2$ & $1/2$ & $36.58\% + 0.4127\%$  \\
 $4/7 $& $3/7$ & $40.88\% + 0.1925\%$ 
\end{tabular}
\caption{\% according to sizes of $\theta$}
\label{my-label}
\end{table}
Since $\psi_2$ is the perturbation of $\psi_1$, it behooves us to take $\theta_1$ as large as possible ($4/7$) at the cost of sacrificing $\theta_2$ to $3/7$ which only adds $0.1925\%$.
\subsection{The smoothing argument}
The idea of smoothing the mean value integrals was introduced in \cite{bcy,youngsimple} and it helps substantially in our calculations. Let $w(t)$ be a smooth function satisfying the following properties:
\begin{enumerate}
\item[(a)] $0 \le w(t) \le 1$ for all $t \in \R$,
\item[(b)] $w$ has compact support in $[T/4,2T]$,
\item[(c)] $w^{(j)}(t) \ll_j \Delta^{-j}$, for each $j=0,1,2,\cdots$ and where $\Delta = T/L$.
\end{enumerate}
This allows us to re-write Theorem \ref{unsmoothed} as follows.
\begin{theorem} \label{smoothed}
Suppose that $\theta_1 = 1/2 - \varepsilon$ and $\theta_2 = 1/2 - \varepsilon$ for $\varepsilon > 0 $ small. For any $w$ satisfying conditions \textnormal{(a)}, \textnormal{(b)} and \textnormal{(c)} and $\sigma_0 = 1/2-R/L$,
\[
\int_{-\infty}^{\infty} w(t)|V \psi(\sigma_0 + it)|^2 dt = c(P_1,P_{\ell},Q,R,\theta_1,\theta_2)\widehat{w}(0) + O(T/L),
\]
uniformly for $R \ll 1$, where $c(P_1,P_{\ell},Q,R,\theta_1,\theta_2) = c_{11} + 2c_{12} + c_{22}$ and the $c_{ij}$ are given by \eqref{c11formula}, \eqref{c12formula} and \eqref{c22formula}.
\end{theorem}
How to deal with a two-piece mollifier was explained in \cite{bcy,feng}. In \cite{rrz01} a 4-piece mollifier was studied. The idea is to open the square in the integrand to get
\begin{align}
\int |V\psi|^2 &= \int |V\psi_1|^2 + \int |V|^2\psi_1 \overline{\psi_2} + \int |V|^2\overline{\psi_1} \psi_2 + \int |V\psi_2|^2 \nonumber \\
&=I_{11} + I_{12} + \overline{I_{12}} + I_{22}. \nonumber
\end{align}
We will compute these integrals in the next sections. The integral $I_{12}$ is asymptotically real, thus $I_{21}$ follows from $I_{12}$, i.e. $I_{12} \sim \overline{I_{21}}$.
\subsection{The main terms}
The main terms coming from integrals $I_{11}$, $I_{12}$ and $I_{22}$ are now stated as theorems.
\begin{theorem}[Conrey] \label{c11theorem}
Suppose $\theta_1 < 4/7$. Then
\[
\int_{-\infty}^{\infty} w(t)|V \psi_1(\sigma_0+it)|^2 dt \sim c_{11}(P_1,Q,R,\theta_1)\widehat{w}(0) + O(T/L)
\]
uniformly for $R \ll 1$, where
\begin{align} \label{c11formula}
{c_{11}}({P_1},Q,R,{\theta _1}) = 1 + \frac{1}{{{\theta _1}}}\int_0^1 \int_0^1 {{e^{2Rv}}{{\left( {\frac{d}{{dx}}{e^{R\theta x}}{P_1}(x + u)Q(v + \theta x){|_{x = 0}}} \right)}^2}dudv}  .
\end{align}
\end{theorem}
Let $(\ell)_k =  \ell(\ell-1)\cdot\ldots\cdot(\ell-k+1)$ denote the Pochhammer symbol.
\begin{theorem} \label{c12theorem}
Suppose $\theta_1 + \theta_2 = 1 - \varepsilon$ where $\theta_1 < 4/7$ and $\theta_2 < 1/2$. Then
\[
\int_{-\infty}^{\infty} w(t)V \psi_1 \overline{\psi_2}(\sigma_0+it) dt \sim c_{12}(P_1,P_\ell,Q,R,\theta_1,\theta_2)\widehat{w}(0) + O(T/L)
\]
uniformly for $R \ll 1$, where
\begin{align} \label{c12formula}
  {c_{12}}({P_1},{P_\ell },Q,R,{\theta _1},{\theta _2}) &= \sum\limits_{\ell  = 2}^K {\frac{{(-1)^\ell}}{{(\ell  - 1)!}}} \int_0^1 {{{(1 - u)}^{\ell  - 1}}{P_1}(u){P_\ell }(u)du}  \nonumber \\
   &\quad - \frac{{{\theta _1} - {\theta _2}}}{{{\theta _1}}}\sum\limits_{\ell  = 2}^K {\frac{{(-1)^\ell}}{{\ell !}}} \int_0^1 {{{(1 - u)}^{\ell  - 1}}P{'_1}\left( {1 - (1 - u)\frac{{{\theta _2}}}{{{\theta _1}}}} \right){P_\ell }(u)du}  \nonumber \\
   &\quad + \frac{1}{{{\theta _1}}}\sum\limits_{\ell  = 2}^K {\frac{{(-1)^\ell}}{{\ell !}}} \frac{{{d^2}}}{{dxdy}}\bigg[{e^{R({\theta _1}x + {\theta _2}y)}}  \int_0^1 \int_0^1 {e^{2Rv}}{{(1 - u)}^\ell }   \nonumber \\
   &\quad \times {P_1}\left( {x + 1 - (1 - u)\frac{{{\theta _2}}}{{{\theta _1}}}} \right){P_\ell }(y + u)Q({\theta _2}y + v)Q({\theta _1}x + v)dudv{\bigg|_{x = y = 0}}\bigg]  .  
\end{align}
\end{theorem}
\begin{theorem} \label{c22theorem}
Suppose $\theta_2 < 1/2$. Then
\[
\int_{-\infty}^{\infty} w(t)|V \psi_2(\sigma_0+it)|^2 dt \sim c_{22}(P_\ell,Q,R,\theta_2)\widehat{w}(0) + O(T/L)
\]
uniformly for $R \ll 1$, where
\begin{align} \label{c22formula}
  {c_{22}}(P_{\ell},Q,R,\theta_2) &= \sum\limits_{{\ell _1} = 2}^K {\sum\limits_{{\ell _2} = 2}^K {\sum\limits_{k = 0}^{\min ({\ell _1},{\ell _2})} {{{( - 1)}^{{\ell _1} + {\ell _2} - 2k}}\binom{\ell_1}{k}{{({\ell _2})}_k}} } }  \nonumber \\
   &\quad \times \frac{{{2^{{\ell _1} + {\ell _2} - 2k}}}}{{({\ell _1} + {\ell _2} - 1)!}}\int_0^1 {{{(1 - u)}^{{\ell _1} + {\ell _2} - 1}}{P_{{\ell _1}}}(u){P_{{\ell _2}}}(u)du}  \nonumber \\
   &\quad + \frac{1}{{{\theta _2}}}\sum\limits_{{\ell _1} = 2}^K {\sum\limits_{{\ell _2} = 2}^K {\sum\limits_{k = 0}^{\min ({\ell _1},{\ell _2})} {{{( - 1)}^{{\ell _1} + {\ell _2} - 2k}}\binom{\ell_1}{k}{{({\ell _2})}_k}} } } \frac{{{2^{{\ell _1} + {\ell _2} - 2k}}}}{{({\ell _1} + {\ell _2})!}} \frac{{{d^2}}}{{dxdy}}\bigg[{e^{R{\theta _2}(x + y)}} \nonumber \\
	 &\quad \times \int_0^1 \int_0^1 {{e^{2Rv}}{{(1 - u)}^{{\ell _1} + {\ell _2}}}{P_{{\ell _1}}}(x + u){P_{{\ell _2}}}(y + u)Q(v + {\theta _2}x)Q(v + {\theta _2}y)dudv{\bigg|_{x = y = 0}}}  \bigg].   
\end{align}
\end{theorem}
\begin{remark}
Note that in \cite{feng}, $c_{11}$, $c_{12}$ and $c_{22}$ are all mixed into one single theorem and it is not immediately clear how to separate each individual $c$-term.
\end{remark}
The smoothing argument is helpful because we can easily deduce Theorem \ref{unsmoothed} from Theorem \ref{smoothed} and so on. By having chosen $w(t)$ to satisfy conditions (a), (b) and (c) and in addition to being an upper bound for the characteristic function of the interval $[T/2,T]$, and with support $[T/2-\Delta,T+\Delta]$, we get
\[
\int_{T/2}^T |V \psi(\sigma_0+it)|^2 dt \le c(P_1,P_\ell,Q,R,\theta_1,\theta_2)\widehat{w}(0) + O(T/L).
\]
Note that $\widehat{w}(0)=T/2 + O(T/L)$. We similarly get a lower bound. Summing over dyadic segments gives the full result.
\subsection{The shift parameters $\alpha$ and $\beta$}
Rather than working directly with $V(s)$, we shall instead consider the following three general shifted integrals
\begin{align}
  {I_{11}}(\alpha ,\beta ) &= \int_{ - \infty }^\infty  {w(t)\zeta (\tfrac{1}{2} + \alpha  + it)\zeta (\tfrac{1}{2} + \beta  - it)\overline {{\psi _1}} {\psi _1}({\sigma _0} + it)dt} , \nonumber \\
  {I_{12}}(\alpha ,\beta ) &= \int_{ - \infty }^\infty  {w(t)\zeta (\tfrac{1}{2} + \alpha  + it)\zeta (\tfrac{1}{2} + \beta  - it)\overline {{\psi _1}} {\psi _2}({\sigma _0} + it)dt} , \nonumber \\
  {I_{22}}(\alpha ,\beta ) &= \int_{ - \infty }^\infty  {w(t)\zeta (\tfrac{1}{2} + \alpha  + it)\zeta (\tfrac{1}{2} + \beta  - it)\overline {{\psi _2}} {\psi _2}({\sigma _0} + it)dt}.  \nonumber  
\end{align}
The computation is now reduced to proving the following three lemmas.
\begin{lemma} \label{lemmac11}
We have
\[
I_{11} = c_{11}(\alpha,\beta)\widehat{w}(0) + O(T/L),
\]
uniformly for $\alpha,\beta \ll L^{-1}$, where
\begin{align}
{c_{11}}(\alpha ,\beta ) = 1 + \frac{1}{\theta_1}\frac{{{d^2}}}{{dxdy}} \bigg[ y_1^{ - \beta x - \alpha y}\int_0^1 \int_0^1 {T^{ - v(\alpha  + \beta )}P_1(x + u)P_1(y + u)dudv}  {\bigg|_{x = y = 0}} \bigg ].
\end{align}
\end{lemma}
\begin{lemma} \label{lemmac12}
We have
\[
I_{12} = c_{12}(\alpha,\beta)\widehat{w}(0) + O(T/L),
\]
uniformly for $\alpha,\beta \ll L^{-1}$, where
\begin{align}
c_{12}(\alpha,\beta) 
   =&
   \sum\limits_{\ell  = 2}^K  \frac{{(-1)^\ell}}{{(\ell-1) !}}
   \int_0^1 {(1 - u)}^{\ell-1} P_{1}(u) P_{\ell}(u)du  \nonumber \\
   &- \frac{\theta_1-\theta_2}{\theta_1}\sum\limits_{\ell  = 2}^K  \frac{{(-1)^\ell}}{{\ell !}}\int_0^1 {{(1 - u)}^{\ell}}{P'_1}\left( {1 - (1 - u)\frac{{{\theta _2}}}{{{\theta _1}}}} \right){P_\ell }(u)du \nonumber \\
   &+
     \frac{1}{\theta_1}
    \sum\limits_{\ell  = 2}^K \frac{{(-1)^\ell}}{{\ell !}} \frac{{{d^2}}}{{dxdy}} 
    \bigg[  y_1^{-\beta x}y_2^{-\alpha y} \nonumber \\
       & \times \int_0^1\int_0^1 T^{ - v(\alpha  + \beta )} {{(1 - u)}^{\ell}}{P_1}\left( {x + 1 - (1 - u)\frac{{{\theta _2}}}{{{\theta _1}}}} \right){P_\ell }(y + u)dudv {\bigg|_{x = y = 0}} \bigg].
\end{align}
\end{lemma}
\begin{lemma} \label{lemmac22}
We have
\[
I_{22} = c_{22}(\alpha,\beta)\widehat{w}(0) + O(T/L),
\]
uniformly for $\alpha,\beta \ll L^{-1}$, where
\begin{align}
c_{22}(\alpha,\beta) 
=& 
 \sum\limits_{{\ell _1} = 2}^K \sum\limits_{{\ell _2} = 2}^K  
  \sum\limits_{k = 0}^{\min ({\ell _1},{\ell _2})} ( - 1)^{\ell_1+\ell_2-2k}  \binom{\ell_1}{k} (\ell _2)_k \nonumber \\
  &\times \frac{2^{\ell_1+\ell_2-2k}}{(\ell_1+\ell_2-1)!} \int_0^1 {(1 - u)}^{\ell_1+\ell_2-1} P_{\ell _1}(u) P_{\ell _2}(u) du \nonumber \\
&+\frac{1}{\theta_2}
  \sum\limits_{{\ell _1} = 2}^K \sum\limits_{{\ell _2} = 2}^K  
  \sum\limits_{k = 0}^{\min ({\ell _1},{\ell _2})} {\binom{\ell_1}{k}{{({\ell _2})}_k}}  ( - 1)^{\ell_1+\ell_2-2k} \frac{2^{\ell_1+\ell_2-2k}}{(\ell_1+\ell_2)!}  \nonumber \\
   &\quad \times \frac{{{d^2}}}{{dxdy}}\bigg[y_2^{-\beta x -\alpha y}\int_0^1 \int_0^1T^{ - v(\alpha  + \beta )}{{{(1 - u)}^{\ell_1+\ell_2 }}{P_{{\ell _1}}}(x + u){P_{{\ell _2}}}(y + u)dudv} {\bigg|_{x = y = 0}}\bigg].
\end{align}
\end{lemma}
To get Theorems \ref{c11theorem}, \ref{c12theorem} and \ref{c22theorem} we use the following technique. Let $I_{\star}$ denote either of the integrals in questions, and note that
\[{I_ {\star} } = Q\left( { - \frac{1}{{\log T}}\frac{d}{{d\alpha }}} \right)Q\left( { - \frac{1}{{\log T}}\frac{d}{{d\beta }}} \right){I_{\star}}(\alpha ,\beta ){\bigg|_{\alpha  = \beta  = -R/L}}.\]
Since $I_{\star}(\alpha ,\beta )$ and $c_{\star}(\alpha ,\beta )$ are holomorphic with respect to $\alpha,\beta$ small, the derivatives appearing in the equation above can be obtained as integrals of radii $\asymp L^{-1}$ around the points $-R/L$, using Cauchy's integral formula. Since the error terms hold uniformly on these contours, the same error terms that hold for $I_{\star}(\alpha ,\beta )$ also hold for $I_{\star}$. That the above differential operator on $c_{\star}(\alpha ,\beta )$ does indeed give $c_{\star}$ follows from
\[
Q \bigg( \frac{-1}{\log T} \frac{d}{d\alpha} X^{-\alpha} \bigg) = Q\bigg( \frac{\log X}{\log T}\bigg) X^{-\alpha}.
\]
Note that from the above equation we get
\begin{align}
  Q\left( {\frac{{ - 1}}{{\log T}}\frac{d}{{d\alpha }}} \right)Q\left( {\frac{{ - 1}}{{\log T}}\frac{d}{{d\beta }}} \right)y_1^{ - \beta x}y_2^{ - \alpha y}{T^{ - v(\alpha  + \beta )}} &= Q\left( {\frac{{\log y_2^y{T^v}}}{{\log T}}} \right)Q\left( {\frac{{\log y_1^x{T^v}}}{{\log T}}} \right)y_1^{ - \beta x}y_2^{ - \alpha y}{T^{ - v(\alpha  + \beta )}} \nonumber \\
   &= Q({\theta _2}y + v)Q({\theta _1}x + v)y_1^{ - \beta x}y_2^{ - \alpha y}{T^{ - v(\alpha  + \beta )}}, \nonumber
\end{align}
as well as
\begin{align}
  Q\left( {\frac{{ - 1}}{{\log T}}\frac{d}{{d\alpha }}} \right)Q\left( {\frac{{ - 1}}{{\log T}}\frac{d}{{d\beta }}} \right) y_2^{-\beta x - \alpha y} T^{-v(\alpha+\beta)} &= Q\left( {\frac{{\log y_2^y{T^v}}}{{\log T}}} \right){(y_2^y{T^v})^{ - \alpha }}Q\left( {\frac{{\log y_2^x{T^v}}}{{\log T}}} \right){(y_2^x{T^v})^{ - \beta }} \nonumber \\
   &= Q({\theta _2}y + v)Q({\theta _2}x + v)y_2^{ - \beta x - \alpha y}{T^{ - v(\alpha  + \beta )}}. \nonumber  
\end{align}
Hence using the differential operators $Q((-1/\log T) d/d\alpha)$ and $Q((-1/\log T) d/d\beta)$ on the last line of $c_{12}(\alpha,\beta)$ we get
\begin{align}
\frac{{{d^2}}}{{dxdy}}\bigg[y_1^{ - \beta x}y_2^{ - \alpha y}\int_0^1 \int_0^1 {T^{ - v(\alpha  + \beta )}}{{(1 - u)}^\ell }&{P_1}\left( {x + 1 - (1 - u)\frac{{{\theta _2}}}{{{\theta _1}}}} \right) \nonumber \\
&{P_\ell }(y + u)Q({\theta _2}y + v)Q({\theta _1}x + v)dudv{\bigg|_{x = y = 0}} \bigg] . \nonumber
\end{align}
Theorem \ref{c12theorem} then follows by setting $\alpha = \beta = -R/L$ and using $T^{z/L}=T^{z/\log T}=e^z$. Similarly, when we use the differential operators $Q((-1/\log T) d/d\alpha)$ and $Q((-1/\log T) d/d\beta)$ on the last line of $c_{22}(\alpha,\beta)$ it becomes
\[\frac{{{d^2}}}{{dxdy}}\bigg[{e^{R{\theta _2}(x + y)}}\int_0^1 \int_0^1 {{e^{2Rv}}{{(1 - u)}^{{\ell _1} + {\ell _2}}}{P_{{\ell _1}}}(x + u){P_{{\ell _2}}}(y + u)Q(v + {\theta _2}x)Q(v + {\theta _2}y)dudv{\bigg|_{x = y = 0}}} \bigg].\]
The same substitutions yield Theorem \ref{c22theorem}.
\section{Preliminary results}
\subsection{Results from complex analysis}
The following results are needed throughout the paper.
\begin{lemma} \label{lemmaAFE}
Suppose that $w(t)$ satisfies conditions \textnormal{(a)}, \textnormal{(b)} and \textnormal{(c)} and that $h$, $k$ are positive integers with $hk \le T^{2 \theta}$ with $\theta < 1/2$, and $\alpha,\beta \ll L^{-1}$. Moreover, set
\[ 
{g_{\alpha ,\beta }}(s,t) = {\pi ^{ - s}}\frac{{\Gamma (\tfrac{1}{2}(\tfrac{1}{2} + \alpha  + s + it))\Gamma (\tfrac{1}{2}(\tfrac{1}{2} + \beta  + s - it))}}{{\Gamma (\tfrac{1}{2}(\tfrac{1}{2} + \alpha  + it))\Gamma (\tfrac{1}{2}(\tfrac{1}{2} + \beta  - it))}},
\]
as well as
\[{X_{\alpha ,\beta ,t}} = {\pi ^{\alpha  + \beta }}\frac{{\Gamma (\tfrac{1}{2}(\tfrac{1}{2} - \alpha  - it))\Gamma (\tfrac{1}{2}(\tfrac{1}{2} - \beta  + it))}}{{\Gamma (\tfrac{1}{2}(\tfrac{1}{2} + \alpha  + it))\Gamma (\tfrac{1}{2}(\tfrac{1}{2} + \beta  - it))}}.\]
Then one has
\begin{align}
  \int_{ - \infty }^\infty  w(t){{\left( {\frac{h}{k}} \right)}^{ - it}} &\zeta (\tfrac{1}{2} + \alpha  + it)\zeta (\tfrac{1}{2} + \beta  - it)dt  = \sum\limits_{hm = kn} {\frac{1}{{{m^{1/2 + \alpha }}{n^{1/2 + \beta }}}}\int_{ - \infty }^\infty  {{V_{\alpha ,\beta }}(mn,t)w(t)dt} }  \nonumber \\
   &+ \sum\limits_{hm = kn} {\frac{1}{{{m^{1/2 - \beta }}{n^{1/2 - \alpha }}}}\int_{ - \infty }^\infty  {{V_{ - \beta , - \alpha }}(mn,t){X_{\alpha ,\beta ,t}}w(t)dt} }  + {O_{A}}({T^{ - A}}), \nonumber  
\end{align}
where 
\[
V_{\alpha ,\beta }(x,t) = \frac{1}{{2\pi i}}\int_{(1)} {\frac{{G(s)}}{s}{g_{\alpha ,\beta }}(s,t){x^{ - s}}ds}, \quad G(s)=e^{s^2}p(s) \quad \textnormal{and} \quad p(s)=\frac{(\alpha+\beta)^2-(2s)^2}{(\alpha+\beta)^2}.
\]
\end{lemma}
\begin{proof}
See Lemma 5 of \cite{youngsimple}. They key point is that non-diagonal terms $hm \ne kn$ can safely be absorbed in the error terms.
\end{proof}
\begin{lemma} \label{integralwithlogderivative}
Suppose $0<\delta \asymp L^{-1}$, $\beta \ll L^{-1}$ and $\beta<\delta$. For some $\nu \asymp (\log \log y)^{-1}$ we have
\begin{align}
  \Upsilon &:= \frac{1}{{2\pi i}}\int_{(\delta )} {\frac{1}{{\zeta (1 + \beta  + u)}}{{\left( {\frac{{\zeta '}}{\zeta }(1 + \beta  + u)} \right)}^{\ell  - r}}{{\left( {\frac{{{y_{\star}}}}{n}} \right)}^u}\frac{{du}}{{{u^{j + 1}}}}}  \nonumber \\
  &= {( - 1)^{\ell  - r}}\frac{1}{{2\pi i}}\oint {{{(\beta  + u)}^{1 - \ell  + r}}{{\left( {\frac{{{y_{\star}}}}{n}} \right)}^u}\frac{{du}}{{{u^{j + 1}}}}}  + O({L^{{\ell-r-2+j}}}) + O\bigg( {{{\bigg( {\frac{{{y_{\star}}}}{n}} \bigg)}^{ - \nu }}{L^\varepsilon }} \bigg), \nonumber  
\end{align}
where $y_{\star}\geq n>0$ and the contour is a circle of radius one enclosing the origin and $-\beta$.
\end{lemma}
\begin{proof}
This follows a similar procedure to Lemma 6.1 of \cite{bcy} where the zero-free region of $\zeta$ is used. Let $Y=o(T)$ be a large parameter to be chosen later. By Cauchy's theorem, $\Upsilon$ is equal to the sum of residues at $u=0$ and $u=-\beta$ plus integrals over the line segments $\gamma_1 = \{ s=it : t \in \R, |t| \ge Y \}$, $\gamma_2 = \{ s=\sigma \pm iY: -c/\log Y \le \sigma \le 0 \}$, and $\gamma_3 = \{ s= - c/\log Y + it: |t| \le Y \}$, where $c$ is some fixed positive positive constant such that $\zeta(1+\beta+u)$ has no zeros in the region on the right-hand side of the contour determined by the $\gamma_i$'s.
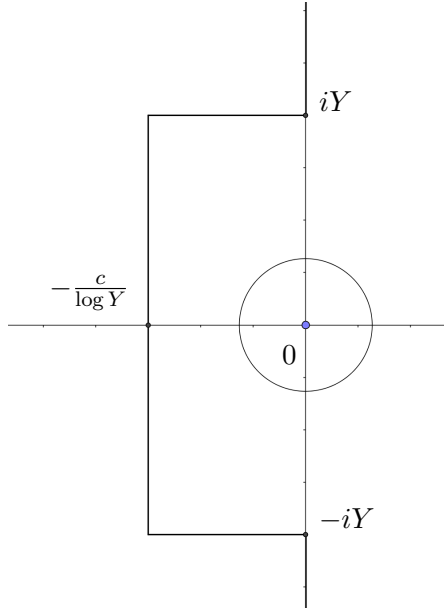
\begin{figure}[ht!]
\centering
\includegraphics[width=.35\textwidth]{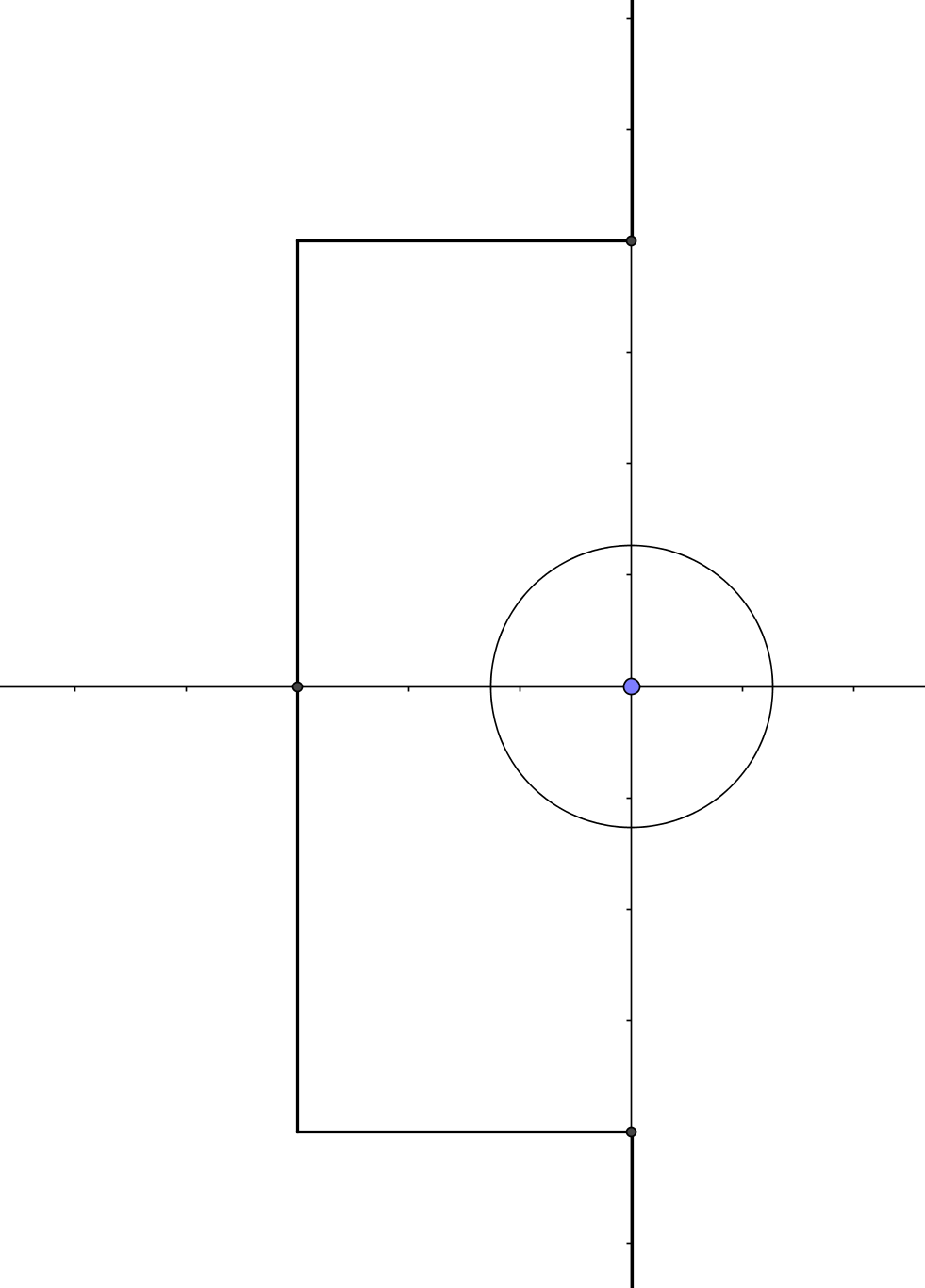}
   \put(-61,92){\mbox{$0$}}
   \put(-47,188){\mbox{$iY$}}
    \put(-47,30){\mbox{$-iY$}}
   \put(-149,118){\mbox{$-\frac{c}{\log Y}$}}
 \caption{Curve $\gamma$ in the proof of Lemma~\ref{integralwithlogderivative}.}
 \label{fig:curve1}
\end{figure}
Another requirement on $c$ is that the estimate (see \cite[Theorem 3.11]{titchmarsh}) $1/\zeta(\sigma+it) \ll \log(2+|t|)$ holds in this region and $\zeta'/\zeta(\sigma+it) \ll \log(4+|t|)$ (see \cite[Theorem 6.7]{montvau}). Then, one has
\[
\int_{\gamma_1} \ll {\int_Y^\infty \frac{\log(t)^{1+\ell-r}}{t^{j+1}}\, dt} \ll
{\frac{\log(Y)^{1+\ell-r}}{Y^{j}}} ,
\]
since $j \ge 3$. Moreover, since $n\leq y_{\star}$
\[
\int_{\gamma_2} \ll {\int_{-c/\log Y}^0 \log(Y)^{1+\ell-r} \left(\frac{y_{\star}}{n}\right)^x \frac{1}{Y^{j+1}} \, dx} 
 \ll {\frac{\log(Y)^{\ell-r}}{Y^{j+1}}} ,
\]
and finally
\[
\int_{\gamma_3} \ll {\int_{-Y}^{Y}  \log(4+|t|)^{\ell-r+1} \frac{(y_{\star}/n)^{-c/\log Y}}{c^2/\log^2 Y +t^2} \, dt}
{\ll \log(Y)^{\ell-r+j}(y_{\star}/n)^{-c/\log Y}} .
\]
Appropriately choosing $Y \asymp (\log y_{\star})$ gives an error of size $O({(\log \log y_{\star})^{\ell-r+j}})= O(\log y_{\star})$. The next step is to sum the residues. This sum can now be expressed as
\[\frac{1}{{2\pi i}}\oint {\frac{1}{{\zeta (1 + \beta  + u)}}{{\left( {\frac{{\zeta '}}{\zeta }(1 + \beta  + u)} \right)}^{\ell  - r}}{{\left( {\frac{{{y_{\star}}}}{n}} \right)}^u}\frac{{du}}{{{u^{j + 1}}}}}, \]
where the contour is now a small circle $\Omega$ of radius $\asymp 1/L$ around the origin such that $-\beta \in \Omega$. Since the radius of the circle is tending to zero, we can use the Laurent expansions
\[\frac{1}{{\zeta (s)}} = s - 1 + O({(s - 1)^2}) \quad \textnormal{and} \quad \frac{{\zeta '}}{\zeta }(s) =  - \frac{1}{{s - 1}} + \gamma  + O(|s - 1|),\]
to finally obtain
\begin{align*}
&\frac{1}{{2\pi i}}\oint {\frac{1}{{\zeta (1 + \beta  + u)}}{{\left( {\frac{{\zeta '}}{\zeta }(1 + \beta  + u)} \right)}^{\ell  - r}}{{\left( {\frac{{{y_{\star}}}}{n}} \right)}^u}\frac{{du}}{{{u^{j + 1}}}}} \nonumber \\
&=
\frac{1}{2\pi i} \oint  (\beta  + u + O(u^2)) \left(\frac{-1}{\beta  + u}+O(1)\right)^{\ell - r} \left( {\frac{{{y_{\star}}}}{n}} \right)^u \frac{du}{u^{j + 1}}.
\end{align*} 
Using the binomial theorem and a direct estimate gives, we get that the above is equal to
\begin{align*}
(-1)^{\ell-r} \oint  (\beta  + u)^{1-\ell  + r} \left( {\frac{{{y_{\star}}}}{n}} \right)^u \frac{du}{u^{j + 1}} + O(L^{j+\ell-r-2}),
\end{align*} 
which is the desired main term of the lemma.
\end{proof}
\noindent This integral can be computed exactly. To do this, note that for any integer $k \ge 1$, one has
\[
q^u (\beta + u)^k = \frac{d^k}{dy^k} e^{\beta y}(e^y q)^u \bigg|_{y=0}.
\]
Hence, one arrives at and where we temporarily set $q = y_{\star}/n$
\begin{align} \label{finalresultlemmaintegrallogderivative}
  \Upsilon &= {( - 1)^{\ell  - r}}\frac{1}{{2\pi i}}\oint {\frac{{{d^{1 - \ell  + r}}}}{{d{y^{1 - \ell  + r}}}}{e^{\beta y}}{{({e^y}q)}^u}{\bigg|_{y = 0}}\frac{{du}}{{{u^{j + 1}}}}}  = {( - 1)^{\ell  - r}}\frac{{{d^{1 - \ell  + r}}}}{{d{y^{1 - \ell  + r}}}}{e^{\beta y}}\frac{1}{{2\pi i}}\oint {{{({e^y}q)}^u}{|_{y = 0}}\frac{{du}}{{{u^{j + 1}}}}}  \nonumber \\
   &= \frac{{{{( - 1)}^{\ell  - r}}}}{{j!}}\frac{{{d^{1 - \ell  + r}}}}{{d{y^{1 - \ell  + r}}}}{e^{\beta y}}{\left( {y + \log \frac{{y_{\star}}}{n}} \right)^j}{\bigg|_{y = 0}},  
\end{align} 
by Cauchy's integral theorem.
\subsection{Combinatorial results}
When computing the crossterm of $\psi_1$ and $\psi_2$ the following result will be needed. This generalizes \cite[Lemma 8]{feng} which is the particular case $h_1=h_2=h$.
\begin{lemma}
\label{lemmacombinatorics}
 For $h_1$ and $h_2$ square-free, we have
 \begin{align}
  \mathcal{Q}(\ell_1,\ell_2) 
  &:= \sum\limits_{\substack{{p_1}{p_2} \cdots p_{\ell _1}|h_1}} \log {p_1}\log {p_2} \cdots \log p_{\ell _1} 
  \sum\limits_{\substack{{q_1}{q_2} \cdots {q_{\ell_2}}|{h_2}}} \log {q_1}\log {q_2} \cdots \log q_{\ell_2} \nonumber \\
   &=
   \sum_{k=0}^{\min\{\ell_1,\ell_2\}}k!\,\binom{\ell_1}{k} \binom{\ell_2}{k}
   \sum_{\substack{p_1p_2\cdots p_{k} q_1q_2\cdots q_{\ell_1-k} r_1r_2\cdots r_{\ell_2-k}| h_1h_2 \\ p_1p_2\cdots p_{k}|\gcd(h_1,h_2)\\q_{1}\cdots q_{\ell_1-k}|h_1\\r_{1}\cdots r_{\ell_2-k}|h_2}} \nonumber \\
   &\quad \times \bigg(\prod_{f=1}^k\log^2 {p_f}\bigg) \bigg(\prod_{f=1}^{\ell_1-k}\log {q_f}\bigg) \bigg(\prod_{f=1}^{\ell_2-k}\log {r_f}\bigg) , \nonumber
 \end{align}
Here the $p$'s, the $q$'s and the $r$'s are all distinct primes. 
\end{lemma}
\begin{proof}
We may write
  \begin{align*}
  \mathcal{Q}(\ell_1,\ell_2) &= \sum\limits_{\substack{{p_1}{p_2} \cdots p_{\ell _1}|h_1\\p_a \neq p_b }} \log {p_1}\log {p_2} \cdots \log p_{\ell _1} 
  \sum\limits_{\substack{{q_1}{q_2} \cdots {q_{\ell_2}}|{h_2}\\ q_a \neq q_b}} \log {q_1}\log {q_2} \cdots \log q_{\ell_2} \\
  &= \ell_1! \ell_2!
  \sum\limits_{\substack{{p_1}{p_2} \cdots p_{\ell _1}|h_1\\p_1 < p_2<\cdots < p_{\ell_1} }} \log {p_1}\log {p_2} \cdots \log p_{\ell _1} 
  \sum\limits_{\substack{{q_1}{q_2} \cdots {q_{\ell_2}}|{h_2}\\ q_1 < q_2<\cdots < q_{\ell_2}}} \log {q_1}\log {q_2} \cdots \log q_{\ell_2} \\
   &= \ell_1! \ell_2!
     \sum_{k=0}^{\min\{\ell_1,\ell_2\}} 
     \sum_{\substack{p_1p_2\cdots p_{k} q_1q_2\cdots q_{\ell_1-k} r_1r_2\cdots r_{\ell_2-k}| h_1h_2 
          \\ p_1p_2\cdots p_{k}|\gcd(h_1,h_2),\,q_{1}\cdots q_{\ell_1-k}|h_1,\,r_{1}\cdots r_{\ell_2-k}|h_2\\
          p_1 < p_2<\cdots < p_{k},\, q_1 < q_2<\cdots < q_{\ell_1-k}, \, r_1 < r_2<\cdots < r_{\ell_2-k}}}\nonumber \\
     &\quad \times  
     \bigl(\log^2 {p_1} \cdots \log^2 p_k\bigr) \bigl(\log q_{1}\cdots \log q_{\ell_1-k}\bigr) \bigl(\log r_{1} \cdots \log r_{\ell_2-k}\bigr)\\
      &= 
     \sum_{k=0}^{\min\{\ell_1,\ell_2\}} \frac{\ell_1! \ell_2!}{k! (\ell_1-k)!(\ell_2-k)!}  \sum_{\substack{p_1p_2\cdots p_{k} q_1q_2\cdots q_{\ell_1-k} r_1r_2\cdots r_{\ell_2-k}| h_1h_2 
          \\ p_1p_2\cdots p_{k}|\gcd(h_1,h_2),\,q_{1}\cdots q_{\ell_1-k}|h_1,\,r_{1}\cdots r_{\ell_2-k}|h_2}}
        \nonumber \\
       & \quad \times  \bigl(\log^2 {p_1} \cdots \log^2 p_k\bigr) \bigl(\log q_{1}\cdots \log q_{\ell_1-k}\bigr) \bigl(\log r_{1} \cdots \log r_{\ell_2-k}\bigr).   
 \end{align*}
Using the definition of the binomial coefficient completes the proof.
\end{proof}
\subsection{Generalized von Mangoldt functions and Euler-MacLaurin summations}
\noindent Recall that for a positive integer $k$, the generalized von Mangoldt function $\Lambda_k(n)$ is defined \cite{ivicmangoldt} by the Dirichlet convolution
\[
\Lambda_k(n) = (\mu * \log^k)(n),
\] 
so that $\Lambda_1(n) = \Lambda(n)$. The generating series is
\[
\sum_{n=1}^{\infty} \frac{\Lambda_k(n)}{n^s} = (-1)^k \frac{\zeta^{(k)}}\zeta (s),
\]
for $\real(s)>1$ and here $\zeta^{(k)}$ stands for the $k$-th derivative of $\zeta$ with respect to $s$. By looking at
\[\frac{d}{{ds}}\bigg( {\frac{{{\zeta ^{(k)}}}}{\zeta }(s)} \bigg) = \frac{{{\zeta ^{(k + 1)}}}}{\zeta }(s) - \frac{{\zeta '}}{\zeta }(s)\frac{{{\zeta ^{(k)}}}}{\zeta }(s)\]
for $\real(s)>1$, we see that
\[{\Lambda _{k + 1}}(n) = {\Lambda _k}(n)\log (n) + (\Lambda  * {\Lambda _{k}})(n),\]
and in particular for $k=1$
\begin{align} \label{mangoldt2case}
	{\Lambda _2}(n) = \Lambda (n)\log (n) + (\Lambda  * \Lambda )(n).
\end{align}
\begin{lemma} \label{lemma1PK}
We have for smooth functions $F$ and $G$ in the interval $[0,1]$, $3 \le z \le x$, and $|s| \le (\log x)^{-1}$
	\begin{align*}
	\sum_{n \le z} \frac{\Lambda(n) \log n}{n^{1+s}} F \left( \frac{\log(x/n)}{\log x}\right) H \left( \frac{\log(z/n)}{\log z}\right)  =&  \frac{\log^2 z}{z^s} \int_{0}^{1} (1-u) F \left( 1 - (1-u)\frac{\log z}{\log x}\right) H \left( u\right) z^{us} du \\
	& + O(\log z).
	\end{align*} 
\end{lemma}
\begin{proof}
Start by setting 
\[
	\Psi(z,x) := \sum_{n \le z} \frac{\Lambda(n) \log n}{n^{1+s}} F \left( \frac{\log(x/n)}{\log x}\right) H \left( \frac{\log(z/n)}{\log z}\right) \quad \textnormal{and} \quad \psi(x) := \sum_{n \le x} \Lambda(n).
\]
By applying the Abel summation formula, one gets 
	\begin{align*}
	\Psi(z,x) & = \psi(z) \frac{\log z}{z^{1 + s}} F \left( \frac{\log(x/z)}{\log x}\right) H \left( 0\right) - \int_1^{z} \psi(u) \frac{d}{du} \left( \frac{\log u}{u^{1+s}} F \left( \frac{\log(x/u)}{\log x}\right) H \left( \frac{\log(z/u)}{\log z}\right) \right) du \\
	& = - \int_1^{z} \psi(u) \frac{d}{du} \left( \frac{\log u}{u^{1+s}} F \left( \frac{\log(x/u)}{\log x}\right) H \left( \frac{\log(z/u)}{\log z}\right) \right) du + O \left( \log z\right) \\
	& = - \int_1^{z} \psi(u) \frac{1 - (1+s)\log u}{u^{2+s}} F \left( \frac{\log(x/u)}{\log x}\right) H \left( \frac{\log(z/u)}{\log z}\right)  du \\
	&\quad - \int_1^{z} \psi(u) \frac{\log u}{u^{1+s}} \left( \frac{d}{du} F\left( \frac{\log(x/u)}{\log x}\right) \right) H \left( \frac{\log(z/u)}{\log z}\right)  du \\
	&\quad - \int_1^{z} \psi(u)  \frac{\log u}{u^{1+s}} F \left( \frac{\log(x/u)}{\log x}\right) \left( \frac{d}{du} H \left( \frac{\log(z/u)}{\log z}\right)  \right) du + O \left( \log z\right) \\
	& = \frac{\log^2 z (1 + s)}{z^s}\int_0^{1} \psi(z^{1-b}) (1-b) F \left( 1 - (1-b) \frac{\log z}{\log x}\right) H \left( b\right) z^{bs + b - 1} db \\ 
	&\quad + O \left( \log z \int_1^{z} \psi(u) \frac{1}{u^{2+|s|}}  du\right) \\
	&\quad + \frac{1}{\log x} \int_1^{z} \psi(u) \frac{\log u}{u^{2+s}} F' \left( \frac{\log(x/u)}{\log x}\right)  H \left( \frac{\log(z/u)}{\log z}\right)  du \\
	&\quad + \frac{1}{\log z}\int_1^{z} \psi(u)  \frac{\log u}{u^{2+s}} F \left( \frac{\log(x/u)}{\log x}\right) H' \left( \frac{\log(z/u)}{\log z}\right) du + O \left( \log z\right) \\
	& = \frac{\log^2 z (1 +s)}{z^s}\int_0^{1} \psi(z^{1-b}) (1-b) F \left( 1 - (1-b) \frac{\log z}{\log x}\right) H\left( b\right) z^{bs + b -1} db + O \left( \log z \right) \\
	& = \frac{\log^2 z}{z^s}\int_0^{1} \psi(z^{1-b}) (1-b) F \left( 1 - (1-b) \frac{\log z}{\log x}\right) H\left( b\right) z^{bs + b -1} db \\
	&\quad + O \left( \log z\int_0^{1} \psi(z^{1-b}) (1-b) z^{bs + b -1} db \right) + O \left( \log z \right) \\
	& = \frac{\log^2 z}{z^s}\int_0^{1} (1-b) F \left( 1 - (1-b) \frac{\log z}{\log x}\right) H\left( b\right) z^{bs} db +  O \left( \log z\int_0^{1} (1-b) z^{bs} db \right) + O \left( \log z \right) \\
	& = \frac{\log^2 z}{z^s}\int_0^{1} (1-b) F \left( 1 - (1-b) \frac{\log z}{\log x}\right) H\left( b\right) z^{bs} db + O \left( \log z \right),
	\end{align*}
since $\psi(x) = x + O(x \exp(- c \sqrt{\log x}))$ for $c > 1$ by the prime number theorem with remainder, see e.g. \cite{titchmarsh}.
\end{proof}
\begin{lemma} \label{lemma36rrz01}
We have for smooth functions $F$ and $G$ in the interval $[0,1]$, $3 \le z \le x$, and $|s| \le (\log x)^{-1}$
\begin{align}
  \sum\limits_{n \leqslant z} {\frac{{({d_k} * {\Lambda ^{ * l}})}}{{{n^{1 + s}}}}} &F\left( {\frac{{\log x/n}}{{\log x}}} \right)H\left( {\frac{{\log z/n}}{{\log z}}} \right) \nonumber \\
  &= \frac{{{{(\log z)}^{k + l}}}}{{(k + l - 1)!{z^s}}}\int_0^1 {{{(1 - u)}^{k + l - 1}}F\left( {1 - (1 - u)\frac{{\log z}}{{\log x}}} \right)H(u){z^{us}}du}  + O({(\log 3z)^{k + l - 1}}), \nonumber 
\end{align}
where $d_k(n)$ denotes the number of ways an integer $n$ can be written as a product of $k \ge 2$ fixed factors. Note that $d_1(n)=1$ and $d_2(n)=d(n)$, the number of divisors of $n$.
\end{lemma}
\begin{proof}
This can be proved by using induction over $\ell$ and Euler-Maclaurin summation. One starts with $\ell=0$ and then uses \cite[Lemma 4.4]{bcy}. The exact details can be found in \cite[Lemma 3.6]{rrz01}. 
\end{proof}
\begin{lemma} \label{lemma2PK}
We have for smooth functions $F$ and $G$ in the interval $[0,1]$, $3 \le z \le x$, and $|s| \le (\log x)^{-1}$
	\begin{align*}
	\sum_{n \le z} \frac{(1 * \Lambda^{*a} * \Lambda \log)(n)}{n^{1+s}} &F \left( \frac{\log(x/n)}{\log x}\right) H \left( \frac{\log(z/n)}{\log z}\right)  \\
	= &  \frac{\log^{3 + a} z}{(a+2)! z^s} \int_{0}^{1} (1-u)^{a+2} F \left( 1 - (1-u)\frac{\log z}{\log x}\right) H \left( u\right) z^{us} du + O(\log^{a+2} z).
	\end{align*} 
\end{lemma}
\begin{proof}
Same as in the beginning of the proof of Lemma \ref{lemma36rrz01} but instead we use Lemma \ref{lemma1PK}.
\end{proof}
\begin{lemma} \label{lemmaEM2lambdas}
We have for smooth functions $F$ and $G$ in the interval $[0,1]$, $3 \le z \le x$, and $|s| \le (\log x)^{-1}$
	\begin{align*}
	 \sum_{n \le z} & \frac{(1 * \Lambda^{*a} * \Lambda_2^{*b})(n)}{n^{1+s}} F \left( \frac{\log(x/n)}{\log x}\right) H \left( \frac{\log(z/n)}{\log z}\right)  \\
	= &  2^b \frac{\log^{1 + a+ 2b} z}{(a+ 2b)! z^s} \int_{0}^{1} (1-u)^{a+2b} F \left( 1 - (1-u)\frac{\log z}{\log x}\right) H \left( u\right) z^{us} du + O(\log^{a+ 2b} z).
	\end{align*} 
\end{lemma}
\begin{proof}
This follows by induction on $b$ and by using Lemma \ref{lemma2PK} combined with \eqref{mangoldt2case}.
\end{proof}
\section{Evaluation of the shifted mean value integrals $I_{\star}(\alpha,\beta)$}
\subsection{Proof of Lemma \ref{lemmac11}}
Although this was already explained in \cite{youngsimple}, the mean value integral $I_{22}(\alpha,\beta)$ builds up from $I_{12}(\alpha,\beta)$ which in turn is a refinement of $I_{11}(\alpha,\beta)$. Therefore, careful analysis will repay itself by going over the main points of the evaluation of $I_{11}(\alpha,\beta)$ briefly. For our purposes, we shall illustrate this for $\theta_1 < 1/2$; however, in \cite{conrey89} it was shown that one could take $\theta_1<4/7$. We start by inserting the definition of the mollifier $\psi_1$ in $I_{11}$ so that
\begin{align}
  {I_{11}}(\alpha ,\beta ) &= \int_{ - \infty }^\infty  {w(t)\zeta (\tfrac{1}{2} + \alpha  + it)\zeta (\tfrac{1}{2} + \beta  - it){\psi _1}\overline {{\psi _1}} ({\sigma _0} + it)dt}  \nonumber \\
   &= \int_{ - \infty }^\infty  {w(t)\zeta (\tfrac{1}{2} + \alpha  + it)\zeta (\tfrac{1}{2} + \beta  - it)}  \nonumber \\
   &\quad \times \sum\limits_{h \leqslant {y_1}} {\frac{{\mu (h){h^{ - 1/2}}}}{{{h^{it}}}}} P_1\left( {\frac{{\log {y_1}/h}}{{\log {y_1}}}} \right)\sum\limits_{k \leqslant {y_1}} {\frac{{\mu (k){k^{ - 1/2}}}}{{{k^{ - it}}}}} P_1\left( {\frac{{\log {y_1}/k}}{{\log {y_1}}}} \right)dt \nonumber \\
   &= \sum\limits_{h \leqslant {y_1}} {\sum\limits_{k \leqslant {y_1}} {\frac{{\mu (h)\mu (k)}}{{{{(hk)}^{1/2}}}}P_1\left( {\frac{{\log {y_1}/h}}{{\log {y_1}}}} \right)P_1\left( {\frac{{\log {y_1}/k}}{{\log {y_1}}}} \right)} }  \nonumber \\
   &\quad \times \int_{ - \infty }^\infty  {w(t){{\left( {\frac{h}{k}} \right)}^{ - it}}\zeta (\tfrac{1}{2} + \alpha  + it)\zeta (\tfrac{1}{2} + \beta  - it)} dt. \nonumber  
\end{align}
According to Lemma \ref{lemmaAFE}, we write $I_{11}(\alpha,\beta)=I'_{11}(\alpha,\beta)+I''_{11}(\alpha,\beta)$, where $I'_{11}$ is given by 
\begin{align} \label{Iprime11}
  I'_{11}(\alpha ,\beta ) &= \sum\limits_{h \leqslant {y_1}} {\sum\limits_{k \leqslant {y_1}} {\frac{{\mu (h)\mu (k)}}{{{{(hk)}^{1/2}}}}P_1\left( {\frac{{\log {y_1}/h}}{{\log {y_1}}}} \right)P_1\left( {\frac{{\log {y_1}/k}}{{\log {y_1}}}} \right)} }  \nonumber \\
   &\quad \times \sum\limits_{hm = kn} \frac{1}{{{m^{1/2 + \alpha }}{n^{1/2 + \beta }}}}\int_{ - \infty }^\infty  {{V_{\alpha ,\beta }}(mn,t)w(t)dt} .   
\end{align}
Notice that $I''_{11}(\alpha,\beta)$ is obtained by replacing $\alpha$ with $-\beta$, $\beta$ with $-\alpha$ and multiplying inside the integrand by $X_{\alpha,\beta,t}=T^{-\alpha,\beta}(1+O(L^{-1}))$. In other words,
\[
I_{11}(\alpha,\beta) = I'_{11}(\alpha,\beta) + T^{-\alpha-\beta} I'_{11}(-\beta,-\alpha) +O(T/L).
\]
Let us then look at $I'_{11}$ more closely. Using the Mellin representations
\[
{P_1}[h] = \sum\limits_i {\frac{{{a_i}i!}}{{{{\log }^i}{y_1}}}} \frac{1}{{2\pi i}}\int_{(1)} {{{\left( {\frac{{{y_1}}}{h}} \right)}^s}\frac{{ds}}{{{s^{i + 1}}}}} \quad \textnormal{and} \quad {P_1}[k] = \sum\limits_j {\frac{{{a_j}j!}}{{{{\log }^j}{y_1}}}} \frac{1}{{2\pi i}}\int_{(1)} {{{\left( {\frac{{{y_1}}}{k}} \right)}^u}\frac{{du}}{{{u^{j + 1}}}}} ,
\]
we then get
\begin{align}
  I'_{11}(\alpha ,\beta ) &= \int_{ - \infty }^\infty  {w(t)\sum\limits_{i,j} {\frac{{{a_i}i!{a_j}j!}}{{{{\log }^{i + j}}{y_1}}}} } {\left( {\frac{1}{{2\pi i}}} \right)^3}\int_{(1)} {\int_{(1)} {\int_{(1)} {y_1^{s + u}{g_{\alpha ,\beta }}(z,t)\frac{{G(z)}}{z}} } }  \nonumber \\
   &\quad \times \sum\limits_{hm = kn} {\frac{{\mu (h)\mu (k)}}{{{h^{1/2+s}}{k^{1/2+u}}{m^{1/2 + \alpha  + z}}{n^{1/2 + \beta  + z}}}}} dz\frac{{ds}}{{{s^{i + 1}}}}\frac{{du}}{{{u^{j + 1}}}}dt. \nonumber
\end{align}
We now evaluate the arithmetical sum $S=\sum_{hm=kn}$ in the integrand. This is done $p$-adically as follows. We denote by $\nu_p(n)$ the number of times the prime number $p$ appears in $n$, and without risk of confusion we write $n'=\nu_p(n)$. This means that
\begin{align}
  S &= \sum\limits_{hm = kn} {\frac{{\mu (h)\mu (k)}}{{{h^{1/2+s}}{k^{1/2+u}}{m^{1/2 + \alpha  + z}}{n^{1/2 + \beta  + z}}}}}  \nonumber \\
   &= \prod\limits_p {\sum\limits_{h' + n' = m' + k'} {\frac{{\mu ({p^{h'}})\mu ({p^{k'}})}}{{{{({p^{h'}})}^{1/2+s}}{{({p^{k'}})}^{1/2+u}}{{({p^{m'}})}^{1/2 + \alpha  + z}}{{({p^{n'}})}^{1/2 + \beta  + z}}}}} }  \nonumber \\
   &= \prod\limits_p {\left( {1 + \frac{1}{{{p^{1 + s + u}}}} - \frac{1}{{{p^{1 + s + \alpha  + z}}}} - \frac{1}{{{p^{1 + u + \beta  + z}}}} + \frac{1}{{{p^{1 + \alpha  + \beta  + 2z}}}} + O({p^{ - 2 + \varepsilon }})} \right)}  \nonumber \\
   &= \frac{{\zeta (1 + s + u)\zeta (1 + \alpha  + \beta  + 2z)}}{{\zeta (1 + s + \alpha  + z)\zeta (1 + u + \beta  + z)}}{A_{\alpha ,\beta }}(s,u,z), \nonumber 
\end{align}
where the arithmetical factor $A_{\alpha,\beta}(s,u,z)$ is given by an absolutely convergent Euler product in some product of half-planes containing the origin. It will be important to remark that when $\alpha = \beta = 0$ and $s=u=z$ we have
\begin{align} \label{arithmeitcal1}
{A_{0,0}}(z,z,z) = \sum\limits_{hm = kn} {\frac{{\mu (h)\mu (k)}}{{{h^{1/2 + z}}{k^{1/2 + z}}{m^{1/2 + z}}{n^{1/2 + z}}}}}  = \sum\limits_{hm = kn} {\frac{{\mu (h)\mu (k)}}{{{{(hkmn)}^{1/2 + z}}}}}  = 1,
\end{align}
for all $z$, by the M\"{o}bius inversion formula. Inserting this into $I'_{11}$ we get
\begin{align}
  I'_{11}(\alpha ,\beta ) &= \int_{ - \infty }^\infty  {w(t)\sum\limits_{i,j} {\frac{{{a_i}i!{a_j}j!}}{{{{\log }^{i + j}}{y_1}}}} } {\left( {\frac{1}{{2\pi i}}} \right)^3}\int_{(1)} {\int_{(1)} {\int_{(1)} {y_1^{s + u}{g_{\alpha ,\beta }}(z,t)\frac{{G(z)}}{z}} } }  \nonumber \\
   &\quad \times \frac{{\zeta (1 + s + u)\zeta (1 + \alpha  + \beta  + 2z)}}{{\zeta (1 + s + \alpha  + z)\zeta (1 + u + \beta  + z)}}{A_{\alpha ,\beta }}(s,u,z)dz\frac{{ds}}{{{s^{i + 1}}}}\frac{{du}}{{{u^{j + 1}}}}dt . \nonumber
\end{align}
Now we deform the path of integration to $\real(z)=-\delta + \varepsilon$ where $\delta>0$ small and $\real(s)=\real(u)=\delta$. By doing this, we pick up a simple pole coming from $1/z$ at $z=0$ only, since $G(z)$ vanishes at the pole of $\zeta(1+\alpha+\beta+2z)$. 
The new path of integration with respect to $z$ contributes an error of the size
\[
\sum_{n \le y_1} \frac{1}{n} \bigg( 1 + \log \frac{y_1}{n} \bigg)^{-2} \ll 1 \ll L^{i+j-2}.
\]
Thus, we end up with
\begin{align}
  I'_{11}(\alpha ,\beta ) &= \int_{ - \infty }^\infty  {w(t)\sum\limits_{i,j} {\frac{{{a_i}i!{a_j}j!}}{{{{\log }^{i + j}}{y_1}}}} } {\left( {\frac{1}{{2\pi i}}} \right)^2}\int_{(\delta)} {\int_{(\delta)} {\mathop {\operatorname{res} }\limits_{z = 0} } } y_1^{s + u}{g_{\alpha ,\beta }}(z,t)\frac{{G(z)}}{z} \nonumber \\
   &\quad \times \frac{{\zeta (1 + s + u)\zeta (1 + \alpha  + \beta  + 2z)}}{{\zeta (1 + s + \alpha  + z)\zeta (1 + u + \beta  + z)}}{A_{\alpha ,\beta }}(s,u,z)\frac{{ds}}{{{s^{i + 1}}}}\frac{{du}}{{{u^{j + 1}}}}dt +O(L^{i+j-2}) \nonumber \\
   &= \widehat w(0)\zeta (1 + \alpha  + \beta )\sum\limits_{i,j} {\frac{{{a_i}i!{a_j}j!}}{{{{\log }^{i + j}}{y_1}}}} {J_{11}}+O(L^{i+j-2}), \nonumber  
\end{align}
where
\[{J_{11}} = {\left( {\frac{1}{{2\pi i}}} \right)^2}\int_{(1)} {\int_{(1)} {y_1^{s + u}} } \frac{{\zeta (1 + s + u)}}{{\zeta (1 + s + \alpha )\zeta (1 + u + \beta )}}{A_{\alpha ,\beta }}(s,u,0)\frac{{ds}}{{{s^{i + 1}}}}\frac{{du}}{{{u^{j + 1}}}}.\]
Using the Dirichlet series representation for $\zeta(1+s+u)$, we can separate the complex variables $s$ and $u$. The next step is to use the Laurent expansion
\begin{align}
  \frac{{{A_{\alpha ,\beta }}(s,u,0)}}{{\zeta (1 + s + \alpha )\zeta (1 + u + \beta )}} &= (\alpha  + s)(\beta  + u){A_{0,0}}(0,0,0) + O({L^{ - 3}}) \nonumber \\
   &= (\alpha  + s)(\beta  + u) + O({L^{ - 3}}) \nonumber  
\end{align}
since $A_{0,0}(z,z,z)=1$ for all $z$, in particular for $z=0$. By the use of Lemma \ref{integralwithlogderivative}, we can deform the line integrals into contour integrals around circles of radius $1$ around the origin. Thus,
\[{J_{11}} = \sum\limits_{n \leqslant {y_1}} {{{\left( {\frac{1}{{2\pi i}}} \right)}^2}\oint {{{\left( {\frac{{{y_1}}}{n}} \right)}^s}\frac{{(s + \alpha )ds}}{{{s^{i + 1}}}}} \oint {{{\left( {\frac{{{y_1}}}{n}} \right)}^s}\frac{{(u + \beta )du}}{{{u^{j + 1}}}}} }  + O({L^{i + j - 2}}).\]
These integrals can be computed by the use of \eqref{finalresultlemmaintegrallogderivative}, so that
\[{J_{11}} = \frac{1}{{i!j!}}\frac{{{d^2}}}{{dxdy}}{e^{\alpha x + \beta y}}\sum\limits_{n \leqslant {y_1}} {\frac{1}{n}{{\left( {x + \log \frac{{{y_1}}}{n}} \right)}^i}{{\left( {y + \log \frac{{{y_1}}}{n}} \right)}^j}} {\bigg|_{x = y = 0}} + O({L^{i + j - 2}}).\]
Let us note that
\[\frac{d}{{dx}}{e^{\alpha x}}\sum\limits_{n \leqslant {y_1}} {\frac{1}{n}{{\left( {x + \log \frac{{{y_1}}}{n}} \right)}^i}} {\bigg|_{x = 0}} = \frac{{{{\log }^i}{y_1}}}{{\log {y_1}}}\frac{d}{{dx}}y_1^{\alpha x}{\left( {x + \frac{{\log ({y_1}/n)}}{{\log {y_1}}}} \right)^i}{\bigg|_{x = 0}}.\]
Now sum over $i$ to get 
\[
{P_1}[n] = \sum\limits_i {{a_i}{{\left( {x + \frac{{\log ({y_1}/n)}}{{\log {y_1}}}} \right)}^i}} 
\]
and similarly over $j$ so that
\begin{align}
  I'_{11}(\alpha,\beta) &= \widehat w(0)\zeta (1 + \alpha  + \beta )\sum\limits_{i,j} {\frac{{{a_i}{a_j}}}{{{{\log }^2}{y_1}}}}  \nonumber \\
   &\quad \times \frac{{{d^2}}}{{dxdy}} \bigg[y_1^{\alpha x + \beta y}\sum\limits_{n \leqslant {y_1}} {\frac{1}{n}{{\left( {x + \frac{{\log ({y_1}/n)}}{{\log {y_1}}}} \right)}^i}{{\left( {y + \frac{{\log ({y_1}/n)}}{{\log {y_1}}}} \right)}^j}} {\bigg|_{x = y = 0}} \bigg] + O(T/L) \nonumber \\
   &= \frac{{\widehat w(0)}}{{(\alpha  + \beta ){{\log }^2}{y_1}}}\frac{{{d^2}}}{{dxdy}}\bigg[y_1^{\alpha x + \beta y} \nonumber \\
   &\quad \times \sum\limits_{n \leqslant {y_1}} {\frac{1}{n}P\left( {x + \frac{{\log ({y_1}/n)}}{{\log {y_1}}}} \right)P\left( {y + \frac{{\log ({y_1}/n)}}{{\log {y_1}}}} \right)} {\bigg|_{x = y = 0}}\bigg] + O(T/L) \nonumber \\
   &= \frac{{\widehat w(0)}}{{(\alpha  + \beta ){{\log }^2}{y_1}}}\frac{{{d^2}}}{{dxdy}}\bigg[y_1^{\alpha x + \beta y} \nonumber \\
   &\quad \times \int_1^{{y_1}} {{r^{ - 1}}P\left( {x + \frac{{\log ({y_1}/r)}}{{\log {y_1}}}} \right)P\left( {y + \frac{{\log ({y_1}/r)}}{{\log {y_1}}}} \right)dr} {\bigg|_{x = y = 0}}\bigg] + O(T/L) \nonumber \\
   &= \frac{{\widehat w(0)}}{{(\alpha  + \beta )\log {y_1}}}\frac{{{d^2}}}{{dxdy}}\bigg[y_1^{\alpha x + \beta y}\int_0^1 {P(x + u)P(y + u)du} {\bigg|_{x = y = 0}}\bigg] + O(T/L). \nonumber  
\end{align}
In the second equality we made use of $\zeta(1+\alpha+\beta)=1/(\alpha+\beta)+O(1)$, in the third equality we used the Euler-MacLaurin formula, and the in the fourth equality we employed the change of variables $r=M^{1-u}$. By adding and subtracting the same quantity we find that
\begin{align} \label{addsubstractI11}
{I_{11}}(\alpha ,\beta ) = [I{'_{11}}(\alpha ,\beta ) + I{'_{11}}( - \beta , - \alpha )] + I{'_{11}}( - \beta , - \alpha )({T^{ - \alpha  - \beta }} - 1) + O(T/L).
\end{align}
For the term in square brackets we have
\begin{align}
  c'_{11}(\alpha ,\beta ) + c{'_{11}}( - \beta , - \alpha ) &= \frac{1}{{(\alpha  + \beta )\log {y_1}}}\int_0^1 {(P'(u) + \alpha P(u)\log {y_1})(P'(u) + \beta P(u)\log {y_1})du}  \nonumber \\
   &\quad - \frac{1}{{(\alpha  + \beta )\log {y_1}}}\int_0^1 {(P'(u) - \beta P(u)\log {y_1})(P'(u) - \alpha P(u)\log {y_1})du}  \nonumber \\
   &= \int_0^1 {2P'(u)P(u)du} = 1. \nonumber  
\end{align}
For the other term in \eqref{addsubstractI11} we have
\begin{align}
  c'_{11}( - \beta , - \alpha )({T^{ - \alpha  - \beta }} - 1) &= \frac{{{T^{ - \alpha  - \beta }} - 1}}{{( - \beta  - \alpha )\log {y_1}}}\frac{{{d^2}}}{{dxdy}}y_1^{ - \beta x - \alpha y}\int_0^1 {P(x + u)P(y + u)du} {\bigg|_{x = y = 0}} \nonumber \\
   &= \frac{{1 - {T^{ - \alpha  - \beta }}}}{{(\alpha  + \beta )\log {y_1}}}\frac{{{d^2}}}{{dxdy}}y_1^{ - \beta x - \alpha y}\int_0^1 {P(x + u)P(y + u)du} {\bigg|_{x = y = 0}} \nonumber \\
   &= \frac{1}{{{\theta _1}}}\frac{{{d^2}}}{{dxdy}}y_1^{ - \beta x - \alpha y}\int_0^1 {\int_0^1 {{T^{ - v(\alpha  + \beta )}}P(x + u)P(y + u)dudv} } {\bigg|_{x = y = 0}}, \nonumber  
\end{align}
by the use of
\begin{align} \label{integraltrick}
\frac{{1 - {T^{ - \alpha  - \beta }}}}{{(\alpha  + \beta )\log {y_1}}} = \frac{1}{{{\theta _1}}}\int_0^1 {{T^{ - v(\alpha  + \beta )}}dv} .
\end{align}
The additional restriction that $|\alpha+\beta| \gg L^{-1}$ is dealt with the holomorphy of $I(\alpha,\beta)$ and $c(\alpha,\beta)$ with $\alpha,\beta \ll L^{-1}$ which implies that the error term is also holomorphic in this region. The maximum modulus principle extends the error term to this enlarged domain. This proves Lemma \ref{lemmac11}.
\subsection{Proof of Lemma \ref{lemmac12}}
This is the term involving Conrey's and Feng's mollifiers. To compute this term, let us follow the same strategy as in $I_{11}(\alpha,\beta)$. We first insert the definitions of $\psi_1$ and $\psi_2$ into the mean value integral $I_{12}$ so that
\begin{align}
I_{12}(\alpha,\beta) &= \int_{-\infty}^{\infty} w(t) \zeta(\tfrac{1}{2}+\alpha+it) \zeta(\tfrac{1}{2}+\beta-it) \overline{\psi_1} \psi_2 (\sigma_0 + it)dt \nonumber \\
   & = \int_{-\infty}^{\infty} w(t) \zeta(\tfrac{1}{2}+\alpha+it) \zeta(\tfrac{1}{2}+\beta-it) \nonumber \\
	 & \quad \times \sum_{h \le y_1} \frac{\mu(h)}{h^{1/2-it}}P_1[h] \sum_{k \le y_2} \frac{\mu(k)}{k^{1/2+it}} \sum_{\ell = 2}^K \sum_{p_1 \cdots p_\ell | k} \frac{\log p_1 \cdots \log p_\ell}{\log ^\ell y_2} P_{\ell}[k] dt \nonumber \\
	 & = \sum_{\ell = 2}^K \sum_{h,k} \frac{\mu(h)\mu(k)}{(hk)^{1/2}} P_1[h] \sum_{p_1 \cdots p_\ell | k} \frac{\log p_1 \cdots \log p_\ell}{\log ^\ell y_2} P_{\ell}[k] \nonumber \\
	 & \quad \times \int_{-\infty}^{\infty} w(t) \zeta(\tfrac{1}{2}+\alpha+it) \zeta(\tfrac{1}{2}+\beta-it) \bigg( \frac{k}{h} \bigg)^{-it} dt. \nonumber
\end{align}
As for $I_{11}(\alpha,\beta)$, we use at this point Lemma~\ref{lemmaAFE} to write $I_{12}(\alpha,\beta) = I'_{12}(\alpha,\beta) + I''_{12}(\alpha,\beta) + E(\alpha,\beta)$, where $I'_{12}(\alpha,\beta)$ and $I''_{12}(\alpha,\beta)$ correspond to the two sums of Lemma~\ref{lemmaAFE} and $E(\alpha,\beta)$ is the error term. Specifically, one has
\begin{align} \label{Iprime12}
  I'_{12}(\alpha ,\beta ) &= 
  \sum_{\ell = 2}^K \sum_{h,k} \frac{\mu(h)\mu(k)}{(hk)^{1/2}} P_1[h] \sum_{p_1 \cdots p_\ell | k} \frac{\log p_1 \cdots \log p_\ell}{\log ^\ell y_2} P_{\ell}[k] \nonumber \\
   &\quad \times \sum\limits_{hm = kn} \frac{1}{{{m^{1/2 + \alpha }}{n^{1/2 + \beta }}}}\int_{ - \infty }^\infty  {V_{\alpha ,\beta }}(mn,t)w(t)dt ,  
\end{align}
and for reasons of symmetry, $I''_{12}(\alpha,\beta)$ can be obtained from $I'_{12}(\alpha,\beta)$ by switching $\alpha$ and $-\beta$ and multiplying by
\[
\bigg( \frac{t}{2\pi} \bigg)^{-\alpha-\beta} = T^{-\alpha-\beta} + O(L^{-1}),
\]
for $t \asymp T$. We thus see that it is enough to compute $I'_{12}(\alpha,\beta)$. The error term is given
\begin{align}
E(\alpha,\beta) &\ll_{A,\theta_1,\theta_\ell} T^{-A} \sum_{\ell = 2}^K \sum_{h,k} \frac{\mu(h)\mu(k)}{(hk)^{1/2}} P_1[h] \sum_{p_1 \cdots p_\ell | k} \frac{\log p_1 \cdots \log p_\ell}{\log ^\ell y_2} P_{\ell}[k] \nonumber \\
& \ll T^{-A} \sum_{\ell = 2}^K \sum_{h \le y_1} \sum_{k \le y_2} \frac{1}{(hk)^{1/2}} \sum_{p_1 \cdots p_\ell | k} 1 \ll T^{-A} \sum_{\ell = 2}^K \sum_{h \le y_1} \sum_{k \le y_2} \frac{1}{(hk)^{1/2}} (d(k))^{\ell}  \nonumber \\ 
& \ll T^{-A} \sum_{h \le y_1} \frac{1}{h^{1/2-\varepsilon}} \sum_{k \le y_2} \frac{1}{k^{1/2-\varepsilon}}  \ll T^{-A} y_1^{1/2-\varepsilon} y_{\ell}^{1/2-\varepsilon}  \nonumber \\
& = T^{-A} T^{\theta_1 (1/2-\varepsilon) \theta_{2}(1/2-\varepsilon)} = T^{-A + (\theta_1+\theta_\ell)/2 - \varepsilon} \nonumber
\end{align}
for any $A>2$. We remark that the above computation works for $\theta_1 + \theta_{2}$ arbitrarily large but the error term $T^{-A}$ coming from Lemma~\ref{lemmaAFE} is only valid for $\theta_1 + \theta_{2} < 1$. The next step is to use the Mellin integral representations of the polynomials $P_1$
\[{P_1}[h] = \sum\limits_i {\frac{{{a_i}}}{{{{\log }^i}{y_1}}}} {(\log ({y_1}/h))^i} = \sum\limits_i {\frac{{{a_i}i!}}{{{{\log }^i}{y_1}}}} \frac{1}{{2\pi i}}\int_{(1)} {{{\left( {\frac{{{y_1}}}{h}} \right)}^s}\frac{{ds}}{{{s^{i + 1}}}}}, \]
and $P_{\ell}$
\[{P_\ell }[k] = \sum\limits_j {\frac{{{b_{\ell ,j}}}}{{{{\log }^j}{y_2}}}} {(\log ({y_2}/k))^j} = \sum\limits_j {\frac{{{b_{\ell ,j}}j!}}{{{{\log }^j}{y_2}}}} \frac{1}{{2\pi i}}\int_{(1)} {{{\left( {\frac{{{y_2}}}{k}} \right)}^u}\frac{{du}}{{{u^{j + 1}}}}} ,\]
and the definition of $V_{\alpha ,\beta}$ in Lemma~\ref{lemmaAFE} to write
\begin{align}
I'_{12}(\alpha,\beta) &= \int_{-\infty}^{\infty} w(t) \sum_{\ell=2}^L \sum_{i,j} \frac{a_i b_{\ell,j} i! j!}{\log^i y_1 \log^{j+\ell}y_2} \nonumber \\
 & \quad \times \sum_{km=hn} \frac{\mu(h)\mu(k)}{(hk)^{1/2} m^{1/2+\alpha} n^{1/2+\beta}} \sum_{p_1 \cdots p_\ell | k} \log p_1 \cdots \log p_\ell \nonumber \\
 & \quad \times \bigg( \frac{1}{2 \pi i} \bigg)^3 \int_{(1)} \int_{(1)} \int_{(1)} \bigg( \frac{y_1}{h} \bigg)^s \bigg( \frac{y_2}{k} \bigg)^u \frac{g_{\alpha,\beta}(z,t)}{(mn)^z} \frac{G(z)}{z} dz \frac{ds}{s^{i+1}} \frac{du}{u^{j+1}} dt. \nonumber
\end{align}
We now have to compute the arithmetical sum $\sum_{km=hn}$. Further details on this procedure can be found in \cite{rrz01}. Let us define
\[
S_\ell = S_{\ell,\alpha,\beta}(s,u,z) = \sum_{km=hn} \frac{\mu(h)\mu(k)}{(hk)^{1/2} m^{1/2+\alpha+z} n^{1/2+\beta+z}} \sum_{p_1 \cdots p_\ell | k} \log p_1 \cdots \log p_\ell.
\]
We start by inverting the order of the sum so that
\begin{align} \label{comparefirst}
S_{\ell} &= (-1)^{\ell} \sum_{\substack{p_i \ne p_j \\ i<j}} \log p_1 \cdots \log p_\ell \sum_{\substack{ hn = p_1 \cdots p_\ell \tilde k m \\ (p_1 \cdots p_\ell ,\tilde k)=1}} \frac{\mu(h) \mu(\tilde k)}{h^{1/2+s}\tilde k^{1/2+u}m^{1/2 + \alpha + z}n^{1/2 +\beta +z}} \frac{1}{(p_1 \cdots p_\ell)^{1/2+u}} \nonumber \\ 
&= (-1)^\ell \sum_{\substack{p_i \ne p_j \\ i<j}} \frac{\log p_1 \cdots \log p_\ell}{(p_1 \cdots p_\ell)^{1/2+u}} \tilde S_{\ell,\alpha,\beta}(s,u,z),
\end{align}
where $k = \tilde k p_1 \cdots p_\ell$ and where we define the inner sum to be
\[
\tilde S_{\ell} = \tilde S_{\ell,\alpha,\beta}(s,u,z) = \sum_{\substack{h, \tilde k, m, n \\ hn=p_1 \cdots p_\ell \tilde k m \\ (p_1 \cdots p_\ell, \tilde k)=1}} \frac{\mu(h) \mu(\tilde k)}{h^{1/2+s}\tilde k^{1/2+u}m^{1/2 + \alpha + z}n^{1/2 +\beta +z}} .
\]
Recall that $\nu_p(n)=n'$ denotes the number of times the prime number $p$ appears in $n$. We can write the above as
\begin{align} \label{comparemiddle}
\tilde S_{\ell} &= \prod_{p \in \{p_1, \cdots, p_\ell \}} \sum_{h'+n'=m'+1} \frac{\mu(p^{h'})}{(p^{h'})^{1/2+s} (p^{m'})^{1/2+\alpha + z} (p^{n'})^{1/2+\beta+z}} \nonumber \\
& \quad \times \prod_{p \notin \{p_1, \cdots, p_\ell \}} \sum_{h'+n'=k'+m'} \frac{\mu(p^{h'})\mu(p^{k'})}{(p^{h'})^{1/2+s} (p^{k'})^{1/2+u} (p^{m'})^{1/2+\alpha + z} (p^{n'})^{1/2+\beta+z}} \nonumber \\
& = \frac{\Pi_1(\alpha,\beta,s,u,z)}{\Pi_2(\alpha,\beta,s,u,z)}{\Pi_3(\alpha,\beta,s,u,z)},
\end{align}
where we define
\begin{align}
\Pi_1(\alpha,\beta,s,u,z) &= \prod_p \sum_{h'+n'=k'+m'} \frac{\mu(p^{h'})\mu(p^{k'})}{(p^{h'})^{1/2+s} (p^{k'})^{1/2+u} (p^{m'})^{1/2+\alpha + z} (p^{n'})^{1/2+\beta+z}} \nonumber \\
& = \prod_p \bigg( 1 + \frac{1}{p^{1+s+u}} - \frac{1}{p^{1+s+\alpha +z}} - \frac{1}{p^{1+u+\beta +z}} +\frac{1}{p^{1+\alpha+\beta +2z}} +O(p^{-2+\varepsilon}) \bigg), \nonumber
\end{align}
as well as
\begin{align}
\Pi_2(\alpha,\beta,s,u,z) &= \prod_{p \in \{ p_1, \cdots, p_\ell \}} \sum_{h'+n'=k'+m'} \frac{\mu(p^{h'})\mu(p^{k'})}{(p^{h'})^{1/2+s} (p^{k'})^{1/2+u} (p^{m'})^{1/2+\alpha + z} (p^{n'})^{1/2+\beta+z}} \nonumber \\
& = \prod_{p \in \{p_1, \cdots, p_\ell \}} \bigg( 1 + \frac{1}{p^{1+s+u}} - \frac{1}{p^{1+s+\alpha +z}} - \frac{1}{p^{1+u+\beta +z}} +\frac{1}{p^{1+\alpha+\beta +2z}} +O(p^{-2+\varepsilon}) \bigg), \nonumber
\end{align}
and finally
\begin{align}
\Pi_3(\alpha,\beta,s,u,z) &= \prod_{p \in \{ p_1, \cdots, p_\ell \}} \sum_{h'+n'=m'+1} \frac{\mu(p^{h'})}{(p^{h'})^{1/2+s}  (p^{m'})^{1/2+\alpha + z} (p^{n'})^{1/2+\beta+z}} \nonumber \\
& = \prod_{p \in \{p_1, \cdots, p_\ell \}} \bigg(\frac{1}{p^{1/2+\beta +z}} - \frac{1}{p^{1/2+s}} + O(p^{-2+\varepsilon}) \bigg). \nonumber
\end{align}
%
 Hence we arrive at the following expression for $\tilde S_\ell$ 
%
\begin{align}
\tilde S_\ell &= \prod_{p} \bigg( 1 + \frac{1}{p^{1+s+u}} -\frac{1}{p^{1+s+\alpha+z}} - \frac{1}{p^{1+u+\beta+z}} + \frac{1}{p^{1+\alpha+\beta+2z}} + O(p^{-2+\varepsilon}) \bigg) \nonumber \\
& = \frac{\zeta(1+s+u)\zeta(1+\alpha+\beta+2z)}{\zeta(1+u+\beta+z)\zeta(1+s+\alpha+z)} A_{\alpha,\beta}(s,u,z), \nonumber 
\end{align}
where the arithmetical factor $A_{\alpha,\beta}(s,u,z)$ is given by an absolutely convergent Euler product in some product of half-planes containing the origin. 
Therefore, when we go back to the expression for $S_{\ell}$ in \eqref{comparefirst}, we obtain the following
\begin{align} \label{comparelast}
S_{\ell} &= \frac{\zeta(1+s+u)\zeta(1+\alpha+\beta+2z)}{\zeta(1+u+\beta+z)\zeta(1+s+\alpha+z)} A_{\alpha,\beta}(s,u,z) (-1)^{\ell} \sum_{\substack{p_i \ne p_j \\ i<j}} \log p_1 \cdots \log p_\ell \nonumber \\
& \quad \times \prod_{p \in \{ p_1, \cdots, p_\ell \}} \frac{E(p)+O(p^{-2+\varepsilon})}{1 - \tfrac{1}{p^{1+s+\alpha+z}} + \tfrac{1}{p^{1+\alpha+\beta+2z}} - E(p) + O(p^{-2 + \varepsilon})},
\end{align}
where
\[
E(p) = \frac{1}{p^{1/2+u}} \bigg( - \frac{1}{p^{1/2+s}} + \frac{1}{p^{1/2+\beta+z}} \bigg) = - \frac{1}{p^{1+s+u}} + \frac{1}{p^{1+\beta+u+z}}. 
\]
At this stage, we compare \eqref{comparelast} in its exact form (that is, with big-$O$ terms replaced by their exact expressions) against \eqref{comparefirst} and \eqref{comparemiddle} in its exact form, and we use the fact that for $\alpha = \beta = 0$ and $s=u=z$, the ratio of zeta functions
\[
\frac{\zeta(1+s+u)\zeta(1+\alpha+\beta+2z)}{\zeta(1+u+\beta+z)\zeta(1+s+\alpha+z)}
\]
reduces to $1$. In other words, reverting the $p$-adic analysis in
\begin{align}
\frac{\zeta(1+s+u)\zeta(1+\alpha+\beta+2z)}{\zeta(1+u+\beta+z)\zeta(1+s+\alpha+z)} & A_{\alpha,\beta}(s,u,z) \nonumber \\
& = \prod_p \sum_{h'+n'=k'+m'} \frac{\mu(p^{h'})\mu(p^{k'})}{(p^{h'})^{1/2+s} (p^{k'})^{1/2+u} (p^{m'})^{1/2+\alpha + z} (p^{n'})^{1/2+\beta+z}}, \nonumber
\end{align}
we find that
\[
\frac{\zeta(1+s+u)\zeta(1+\alpha+\beta+2z)}{\zeta(1+u+\beta+z)\zeta(1+s+\alpha+z)} A_{\alpha,\beta}(s,u,z) = \sum_{hn=km} \frac{\mu(h)\mu(k)}{h^{1/2+s}k^{1/2+u}m^{1/2+\alpha+z}n^{1/2+\beta+z}}.
\]
Following \eqref{arithmeitcal1}, we know that
\[
A_{0,0}(z,z,z)= \sum_{km=hn} \frac{\mu(h)\mu(k)}{(hkmn)^{1/2+z}},
\]
and thus, we find that
\[
A_{0,0}(z,z,z)=1
\]
for all $z$. Let us denote the last part of \eqref{comparelast} by $H_{\ell}$, specifically
\begin{align}
H_\ell &= (-1)^\ell \sum_{\substack{p_i \ne p_j \\ i<j}} \prod_{p \in \{p_1, \cdots, p_\ell \}} (E(p) + O(p^{-2+\varepsilon})) \log p \bigg(1 + E(p) + \frac{1}{p^{1+s+\alpha+z}} - \frac{1}{p^{1+\alpha+\beta+2z}} + O(p^{-2+\varepsilon}) \bigg) \nonumber \\
& = (-1)^\ell \sum_{\substack{p_i \ne p_j \\ i<j}} \prod_{p \in \{p_1, \cdots, p_\ell \}} \bigg( E(p) \log p + O\bigg( \frac{\log p}{p^{2-\varepsilon}} \bigg) \bigg). \nonumber
\end{align}
We now employ the principle of inclusion-exclusion to write
\[
H_\ell = (-1)^\ell \bigg(\sum_{p} E(p) \log p + O\bigg( \frac{\log p}{p^{2-\varepsilon}} \bigg) \bigg)^\ell + \sum_p B(p),
\]
where
\[
B(p) \ll_{\alpha, \beta, s, u, z, \varepsilon} \frac{1}{p^{2-\varepsilon}}.
\]
To end the computation, we must identify the logarithms of the prime numbers with the signature of the von Mangoldt function $\Lambda(n)$ and hence match the resulting expressions to logarithmic derivatives of the Riemann zeta-function by the use of
\[
\frac{\zeta'}{\zeta}(s) = - \sum_{n=1}^\infty {\Lambda(n)}{n^{-s}} = - \sum_p \frac{\log p}{p^s}\bigg(1 -\frac{1}{p^s} \bigg)^{-1} = - \sum_p \frac{\log p}{p^s} + O\bigg( \frac{\log p}{p^{2s}} \bigg),
\]
for $\real(s) > 1$. With this in mind, $H_\ell$ becomes
\begin{align}
H_\ell &= (-1)^\ell \bigg( \frac{\zeta'}{\zeta} (1+s+u) - \frac{\zeta'}{\zeta} (1+\beta+u+z) + O(1) \bigg)^\ell + D(\alpha,\beta,s,u,z) \nonumber \\
& = (-U)^\ell + \sum_{m=0}^{\ell-1} U^m B_m(\alpha,\beta,s,u,z) + D(\alpha,\beta,s,u,z), \nonumber
\end{align}
where $D(\alpha,\beta,s,u,z)$ are terms of smaller order and where
\[
U = - \frac{\zeta'}{\zeta} (1+s+u) + \frac{\zeta'}{\zeta} (1+\beta+u+z).
\]
We also have
that
\[
B_m(\alpha, \beta, s, u, z) \ll_{\alpha,\beta,s,u,z} \sum_p \frac{\log p}{p^{2-\varepsilon}}.
\]
All of these terms are analytic in a larger region of the complex plane, thus we are only interested in the term $U^\ell$. Consequently, the end result of this is that
\begin{align}
I'_{12}(\alpha,\beta) &= \int_{-\infty}^\infty w(t) \sum_{\ell = 2}^K \sum_{i,j} \frac{a_i b_{\ell,j} i! j!}{\log^i y_1 \log^{j+\ell}y_2} \bigg( \frac{1}{2 \pi i}\bigg)^3 \int_{(1)} \int_{(1)} \int_{(1)} \nonumber \\
& \quad \times \frac{\zeta(1+s+u)\zeta(1+\alpha+\beta+2z)}{\zeta(1+u+\beta+z)\zeta(1+s+\alpha+z)} A_{\alpha,\beta}(s,u,z) \bigg( \frac{\zeta'}{\zeta} (1+s+u) - \frac{\zeta'}{\zeta} (1+\beta+u+z) \bigg)^\ell \nonumber \\
& \quad \times (-1)^\ell y_1^s y_2^u  \frac{G(z)}{z} g_{\alpha,\beta}(z,t) \frac{ds}{s^{i+1}} \frac{du}{u^{j+1}}dt. \nonumber
\end{align}
The next step is to deform the path of integration to $\real(z) = -\delta + \varepsilon$ where $\delta > 0$ is small, fixed and $\delta < \varepsilon$ as well as $\real(s) = \real(u) = \delta$. By doing this, we pick up the contribution of the residue of the simple pole of $1/z$ at $z=0$ only, since, as before in the $I_{11}(\alpha,\beta)$ case, $G(z)$ vanishes at the pole of $\zeta(1+\alpha+\beta+2z)$. 
The new path of integration with respect to $z$ contributes
\begin{align} \label{newpatherrorterms}
\ll T^{1+\varepsilon} \bigg( \frac{y_1 y_2}{T} \bigg)^\delta \ll T^{1-\varepsilon}.
\end{align}
by keeping $\theta_1+\theta_2 = 1-\varepsilon$ (since $y_1 = T^{\theta_1}$ and $y_2 = T^{\theta_2}$). We now write
\[
I'_{12}(\alpha,\beta) = I'_{120}(\alpha,\beta) + O(T^{-1+\varepsilon}), 
\]
where $I'_{120}(\alpha,\beta)$ corresponds to the residue at $z=0$. Then
\begin{align} \label{IintoJ}
  I'_{120}(\alpha ,\beta ) &= \int_{ - \infty }^\infty  {w(t)} \sum\limits_{\ell  = 2}^K {\sum\limits_{i,j} {\frac{{{a_i}{b_{\ell ,j}}i!j!}}{{{{\log }^i}{y_1}{{\log }^{j + \ell }}{y_2}}}} }  \nonumber \\
   &\quad \times {\left( {\frac{1}{{2\pi i}}} \right)^2}\int_{(\delta )} \int_{(\delta )} {\mathop {\operatorname{res} }\limits_{z = 0} } \frac{{G(z)}}{z} g_{\alpha,\beta}(z,t) y_1^sy_2^u\frac{{\zeta (1 + s + u)\zeta (1 + \alpha  + \beta  + 2z)}}{{\zeta (1 + s + \alpha  + z)\zeta (1 + u + \beta  + z)}}{A_{\alpha ,\beta }}(s,u,z) \nonumber \\
   & \quad \times {( - 1)^\ell }{\left( {\frac{{\zeta '}}{\zeta }(1 + s + u) - \frac{{\zeta '}}{\zeta }(1 + \beta  + u + z)} \right)^\ell }\frac{{ds}}{{{s^{i + 1}}}}\frac{{du}}{{{u^{j + 1}}}}dt \nonumber \\
   &= \widehat w(0)\zeta (1 + \alpha  + \beta )\sum\limits_{\ell  = 2}^K {(-1)^\ell} {\sum\limits_{i,j} {\frac{{{a_i}{b_{\ell ,j}}i!j!}}{{{{\log }^i}{y_1}{{\log }^{j + \ell }}{y_2}}}} } {J_{12}}, 
\end{align}
where
\begin{align}
J_{12}(\alpha,\beta) &= \bigg( \frac{1}{2 \pi i}\bigg)^2 \int_{(\delta)} \int_{(\delta)} \frac{\zeta(1+s+u) A_{\alpha,\beta}(s,u,0)}{\zeta(1+u+\beta)\zeta(1+s+\alpha)} \bigg( \frac{\zeta'}{\zeta} (1+s+u) - \frac{\zeta'}{\zeta} (1+\beta+u) \bigg)^\ell  y_1^s y_2^u \frac{ds}{s^{i+1}} \frac{du}{u^{j+1}}. \nonumber
\end{align}
Let us now use the binomial theorem to write
\begin{align}
  J_{12}(\alpha ,\beta ) &= {\left( {\frac{1}{{2\pi i}}} \right)^2}\int_{(\delta )} {\int_{(\delta )} {\frac{{\zeta (1 + s + u){A_{\alpha ,\beta }}(s,u,0)}}{{\zeta (1 + \beta  + u)\zeta (1 + \alpha  + s)}}} }  \nonumber \\
   &\quad \times \sum\limits_{r = 0}^\ell  \binom{\ell}{r} {\left( {\frac{{\zeta '}}{\zeta }(1 + \beta  + u)} \right)^{\ell  - r}}{\left( { - \frac{{\zeta '}}{\zeta }(1 + s + u)} \right)^r}y_1^sy_2^u\frac{{ds}}{{{s^{i + 1}}}}\frac{{du}}{{{u^{j + 1}}}} \nonumber \\
   &= {\left( {\frac{1}{{2\pi i}}} \right)^2}\int_{(\delta )} {\int_{(\delta )} {\frac{{{A_{\alpha ,\beta }}(s,u,0)}}{{\zeta (1 + \beta  + u)\zeta (1 + \alpha  + s)}}} } \sum\limits_{r = 0}^\ell  \binom{\ell}{r} {\left( {\frac{{\zeta '}}{\zeta }(1 + \beta  + u)} \right)^{\ell  - r}} \nonumber \\
   &\quad \times \sum\limits_{n = 1}^\infty  {\frac{{(\mathbf{1} * {\Lambda ^{ * r}})(n)}}{{{n^{1 + s + u}}}}} y_1^sy_2^u\frac{{ds}}{{{s^{i + 1}}}}\frac{{du}}{{{u^{j + 1}}}} \nonumber \\
   &= \sum\limits_{n \leqslant \min ({y_1},{y_2})} {\sum\limits_{r = 0}^\ell  \binom{\ell}{r} \frac{{(\mathbf{1} * {\Lambda ^{ * r}})(n)}}{n}} {\left( {\frac{1}{{2\pi i}}} \right)^2} \nonumber \\
   &\quad \times \int_{(\delta )} {\int_{(\delta )} {\frac{A_{\alpha,\beta}(s,u,0)}{{\zeta (1 + \beta  + u)\zeta (1 + \alpha  + s)}}} } {\left( {\frac{{\zeta '}}{\zeta }(1 + \beta  + u)} \right)^{\ell  - r}}{\left( {\frac{{{y_1}}}{n}} \right)^s}{\left( {\frac{{{y_2}}}{n}} \right)^u}\frac{{ds}}{{{s^{i + 1}}}}\frac{{du}}{{{u^{j + 1}}}}, \nonumber 
\end{align}
where we have used the Dirichlet convolution of
\[
\zeta(s) = \sum_{n=1}^\infty \frac{1}{n^s} \quad \textnormal{and} \quad -\frac{\zeta'}{\zeta}(s) = \sum_{n=1}^\infty \frac{\Lambda(n)}{n^s},
\]
for $\real(s) > 1$. Here $\mathbf{1}(n)=1$ for all $n$ denotes the identity function. Next, we take $\delta \asymp L^{-1}$ and bound the integral trivially to get $J_{12} \ll L^{i+j-1}$. 
This means that we can use a Taylor series so that $A_{\alpha,\beta}(s,u,0) = A_{0,0}(0,0,0) + O(|s| + |u|)$ to write $J_{12}(\alpha,\beta)=J'_{12}(\alpha,\beta) + O(L^{i+j-2})$, say. 
We recall that we have shown earlier that $A_{0,0}(z,z,z)=1$ for all $z$, in particular $A_{0,0}(0,0,0)=1$. This implies that the complex variables $s$ and $u$ are now separated as
\[
J'_{12}(\alpha ,\beta ) = \sum\limits_{n \leqslant \min ({y_1},{y_2})} {\sum\limits_{r = 0}^\ell  \binom{\ell}{r} \frac{{(\mathbf{1} * {\Lambda ^{ * r}})(n)}}{n}} {L_{12,1}}{L_{12,2}},
\]
where 
\[{L_{12,1}} = \frac{1}{{2\pi i}}\int_{(\delta )} {\frac{1}{{\zeta (1 + \alpha  + s)}}{{\left( {\frac{{{y_1}}}{n}} \right)}^s}\frac{{ds}}{{{s^{i + 1}}}}} ,\]
and
\begin{align} \label{L2CF}
{L_{12,2}} = \frac{1}{{2\pi i}}\int_{(\delta )} {\frac{1}{{\zeta (1 + \beta  + u)}}{{\left( {\frac{{\zeta '}}{\zeta }(1 + \beta  + u)} \right)}^{\ell  - r}}{{\left( {\frac{{{y_2}}}{n}} \right)}^u}\frac{{du}}{{{u^{j + 1}}}}}.
\end{align}
The first of these two integrals was dealt with in the $I_{11}(\alpha,\beta)$ case and its main term is
\[
L_{12,1} = \frac{1}{2 \pi i} \oint (\alpha+s) \bigg( \frac{y_1}{n} \bigg)^s \frac{ds}{s^{i+1}}  = \frac{1}{{i!}}\frac{d}{{dx}}{e^{\alpha x}}{\left( {x + \log \frac{{{y_1}}}{n}} \right)^i}{\bigg|_{x = 0}}.
\]
For the second integral we will need the following Lemma \ref{integralwithlogderivative} and equation \eqref{finalresultlemmaintegrallogderivative}. Hence, one gets
\begin{align}
  L_{12,2} = \frac{(-1)^{\ell-r}}{2 \pi i} \oint (\beta + u)^{1-\ell+r} \bigg(\frac{y_2}{n}\bigg)^u \frac{du}{u^{j+1}} = \frac{{{{( - 1)}^{\ell  - r}}}}{{j!}}\frac{{{d^{1 - \ell  + r}}}}{{d{y^{1 - \ell  + r}}}}{e^{\beta y}}{\left( {y + \log \frac{{{y_2}}}{n}} \right)^j}{\bigg|_{y = 0}}. \nonumber 
\end{align} 
This means that when we insert these results into $J'_{12}$ we obtain
\begin{align}
  J'_{12}(\alpha ,\beta ) &= \frac{1}{{i!}}\frac{1}{{j!}}\sum\limits_{n \leqslant \min ({y_1},{y_2})} {\sum\limits_{r = 0}^\ell  {{{( - 1)}^{\ell  - r}}\binom{\ell}{r}} \frac{{(\mathbf{1} * {\Lambda ^{ * r}})(n)}}{n}}  \nonumber \\
   &\quad \times \frac{d}{{dx}}\frac{{{d^{1 - \ell  + r}}}}{{d{y^{1 - \ell  + r}}}}{e^{\alpha x + \beta y}}{\left( {x + \log \frac{{{y_1}}}{n}} \right)^i}{\bigg|_{x = 0}}{\left( {y + \log \frac{{{y_2}}}{n}} \right)^j}{\bigg|_{y = 0}} + O(L^{i+j-2}). \nonumber 
\end{align}
By making the changes
\[
x \to \frac{x}{{\log {y_1}}} \quad \textnormal{and} \quad y \to \frac{y}{{\log {y_2}}},
\]
we can write this in the more convenient form
\begin{align}
  J'_{12}(\alpha ,\beta ) &= \frac{{{{\log }^{i - 1}}{y_1}{{\log }^{j - 1}}{y_2}}}{{i!j!}}\sum\limits_{n \leqslant \min ({y_1},{y_2})} {} \sum\limits_{r = 0}^\ell  {} {( - 1)^{\ell  - r}}\binom{\ell}{r}\frac{{(\mathbf{1} * {\Lambda ^{ * r}})(n)}}{n} \nonumber \\
   &\quad \times \frac{d}{{dx}}\frac{{{d^{1 - \ell  + r}}}}{{d{y^{1 - \ell  + r}}}}y_1^{\alpha x}y_2^{\beta y}{\left( {x + \frac{{\log ({y_1}/n)}}{{\log {y_1}}}} \right)^i}{\bigg|_{x = 0}}{\left( {y + \frac{{\log ({y_2}/n)}}{{\log {y_2}}}} \right)^j}{\bigg|_{y = 0}} + O(L^{i+j-2}). \nonumber 
\end{align}
Telescoping back to \eqref{IintoJ} we obtain that
\begin{align}
  I'_{120}(\alpha ,\beta ) &= \frac{{\widehat w(0)}}{{(\alpha  + \beta )\log {y_1}\log {y_2}}}\frac{{{d^2}}}{{dxdy}} \bigg[y_1^{\alpha x}y_2^{\beta y} \nonumber \\
   &\quad \times \sum\limits_{\ell  = 2}^L {(-1)^\ell} \frac{1}{{{{\log }^\ell }{y_2}}}\sum\limits_{r = 0}^\ell  {} {( - 1)^{\ell  - r}}\binom{\ell}{r}\frac{{{d^{r - \ell }}}}{{d{y^{r - \ell }}}} \nonumber \\
   &\quad \times \sum\limits_{n \leqslant \min ({y_1},{y_2})} {} \frac{{(\mathbf{1} * {\Lambda ^{ * r}})(n)}}{n}{P_1}\left( {x + \frac{{\log ({y_1}/n)}}{{\log {y_1}}}} \right){P_\ell }\left( {y + \frac{{\log ({y_2}/n)}}{{\log {y_2}}}} \right){\bigg|_{x = y = 0}} \bigg] + O(T/L), \nonumber 
\end{align}
where the sum over $i$ has been identified to the polynomial $P_1$, and the sum over $j$ to the polynomials $P_{\ell}$. We now perform the summation over $n$ by using Lemma \ref{lemma36rrz01}. To do so, we now set $y_1 \ge y_2$. The lemma yields
\begin{align}
  \sum\limits_{n \leqslant {y_2}} {\frac{{(\mathbf{1} * {\Lambda ^{ * r}})(n)}}{{{n^{1 + s}}}}} & {P_1}\left( {x + \frac{{\log ({y_1}/n)}}{{\log {y_1}}}} \right){P_\ell }\left( {y + \frac{{\log ({y_2}/n)}}{{\log {y_2}}}} \right) \nonumber \\
   &= \frac{{{{\log }^{r+1}}{y_2}}}{{y_2^s}}\int_0^1 (1-u)^r{P_1}\left( {x + 1 - (1 - u)\frac{{{\theta _2}}}{{{\theta _1}}}} \right){P_\ell }(y + u)y_2^{us}du  + O(\log(3y_2)^{r}). \nonumber 
\end{align}
Therefore, the resulting expression for $I'_{120}$ is
\begin{align}
  I'_{120}(\alpha ,\beta ) &= \frac{{\widehat w(0)}}{{(\alpha  + \beta )\log {y_1}\log {y_2}}}\frac{{{d^2}}}{{dxdy}} \bigg[\int_0^1 {}  \nonumber \\
   &\quad \times y_1^{\alpha x}y_2^{\beta y}\sum\limits_{\ell  = 2}^L {(-1)^\ell} \frac{1}{{{{\log }^\ell }{y_2}}}\sum\limits_{r = 0}^\ell  {} {( - 1)^{\ell  - r}}\binom{\ell}{r}\frac{{{d^{r - \ell }}}}{{d{y^{r - \ell }}}} \nonumber \\
   &\quad \times \frac{{{{\log }^{r+1}}{y_2}}}{{r!}}{{(1 - u)}^{r}}{P_1}\left( {x + 1 - (1 - u)\frac{{{\theta _2}}}{{{\theta _1}}}} \right){P_\ell }(y + u)du {\bigg|_{x = y = 0}} \bigg] + O(T/L) \nonumber 
\end{align}
Now we must go back to $I_{12}$. We recall that $I_{12}(\alpha,\beta)$ was formed by adding $I'_{12}(\alpha,\beta)$ and $I''_{12}(\alpha,\beta)$, where $I''_{12}$ is formed by taking $I'_{12}$, switching $\alpha$ and $-\beta$, and then multiplying by $T^{-\alpha-\beta}$. Note that $r\leq \ell$ and thus only the case $r=\ell$ contributes to the main term. Therefore
 \begin{align}
  I'_{12}(\alpha ,\beta ) &= \frac{{\widehat w(0)}}{{(\alpha  + \beta )\log {y_1}}}\sum\limits_{\ell  = 2}^L {(-1)^\ell} \frac{1}{{\ell !}} \frac{{{d^2}}}{{dxdy}} \bigg[  y_1^{\alpha x}y_2^{\beta y}\nonumber \\
   &\times\quad \int_0^1 {} {{(1 - u)}^{\ell}}{P_1}\left( {x + 1 - (1 - u)\frac{{{\theta _2}}}{{{\theta _1}}}} \right){P_\ell }(y + u)du {\bigg|_{x = y = 0}} \bigg] + O(T/L). \nonumber 
\end{align}
We now use 
\begin{align*}
I_{12}(\alpha,\beta) &= I'_{12}(\alpha,\beta) +T^{-\alpha-\beta}I'_{12}(-\beta,-\alpha) +O(T/L)\\
&= \bigl(I'_{12}(\alpha,\beta) +I'_{12}(-\beta,-\alpha)\bigr)+ \bigl(T^{-\alpha-\beta}-1\bigr)I'_{12}(-\beta,-\alpha) +O(T/L).
\end{align*}
We first take a look at the first term in the brackets
 \begin{align*}
  &\frac{d^2}{dxdy}\bigg[\left(y_1^{\alpha x}y_2^{\beta y}-y_1^{-\beta x}y_2^{-\alpha y}\right)
 \int_0^1 {} {{(1 - u)}^{\ell}}{P_1}\left( {x + 1 - (1 - u)\frac{{{\theta _2}}}{{{\theta _1}}}} \right){P_\ell }(y + u)du {\bigg|_{x = y = 0}}\bigg]\\
  =& \
  (\alpha+\beta)\log y_1 \int_0^1  {{(1 - u)}^{\ell}}{P_1}\left( {1 - (1 - u)\frac{{{\theta _2}}}{{{\theta _1}}}} \right){P'_\ell }(y)du\\
  &+ (\alpha+\beta)\log y_2 \int_0^1 {{(1 - u)}^{\ell}}{P'_1}\left( {1 - (1 - u)\frac{{{\theta _2}}}{{{\theta _1}}}} \right){P_\ell }(y)du\\
  =& \
    (\alpha+\beta)\log y_1 \left(\int_0^1  {{(1 - u)}^{\ell}}{P_1}\left( {1 - (1 - u)\frac{{{\theta _2}}}{{{\theta _1}}}}\right){P'_\ell }(y)du +\int_0^1 {{(1 - u)}^{\ell}}{P'_1}\left( {1 - (1 - u)\frac{{{\theta _2}}}{{{\theta _1}}}} \right){P_\ell }(y)du\right)\\
    &- (\alpha+\beta)(\theta_1-\theta_2)\log T \int_0^1 {{(1 - u)}^{\ell}}{P'_1}\left( {1 - (1 - u)\frac{{{\theta _2}}}{{{\theta _1}}}} \right){P_\ell }(y)du
 \end{align*}
Since we had that $P_1(0)=P_\ell(0)=0$ it follows that
\begin{align*}
  &0 = {(1 - u)}^{\ell} P_{1}(u) P_{\ell}(u)\bigg|_{u = 0}^1 
  = 
  \int_0^1 \left({(1 - u)}^{\ell} P_{1}(u) P_{\ell}(u)\right)'du.
   \end{align*}
We can therefore write 
 \begin{align*}
 \ell  \int_0^1 {(1 - u)}^{\ell-1} P_{1}(u) P_{\ell}(u)du
 =&
 \int_0^1  {{(1 - u)}^{\ell}}{P_1}\left( {1 - (1 - u)\frac{{{\theta _2}}}{{{\theta _1}}}}\right){P'_\ell }(u)du \\
 &+\int_0^1 {{(1 - u)}^{\ell}}{P'_1}\left( {1 - (1 - u)\frac{{{\theta _2}}}{{{\theta _1}}}} \right){P_\ell }(u)du.
   \end{align*}
Combining these observations, we see that   
\begin{align*}
   I'_{12}(\alpha,\beta) +I'_{12}(-\beta,-\alpha)
   =&
   \widehat w(0) \sum\limits_{\ell  = 2}^L  \frac{{(-1)^\ell}}{{(\ell-1) !}}
   \int_0^1 {(1 - u)}^{\ell-1} P_{1}(u) P_{\ell}(u)du  \\
   &- \widehat w(0) \frac{\theta_1-\theta_2}{\theta_1}\sum\limits_{\ell  = 2}^L  \frac{{(-1)^\ell}}{{\ell !}}\int_0^1 {{(1 - u)}^{\ell}}{P'_1}\left( {1 - (1 - u)\frac{{{\theta _2}}}{{{\theta _1}}}} \right){P_\ell }(u)du.
\end{align*}
For the expression $(T^{ - \alpha  - \beta }-1)I{'_{12}}( - \beta , - \alpha )$, we use \eqref{integraltrick} to get 
\begin{align*}
  &(T^{-\alpha-\beta}-1)I'_{12}(-\beta,-\alpha)\nonumber \\
  = &
  \frac{{\widehat {w}(0)}}{\theta_1}
 \sum\limits_{\ell  = 2}^L  \frac{{(-1)^\ell}}{{\ell !}} \frac{{{d^2}}}{{dxdy}} 
 \bigg[  y_1^{-\beta x}y_2^{-\alpha y}
     \int_0^1\int_0^1 T^{ - v(\alpha  + \beta )} {{(1 - u)}^{\ell}}{P_1}\left( {x + 1 - (1 - u)\frac{{{\theta _2}}}{{{\theta _1}}}} \right){P_\ell }(y + u)dudv {\bigg|_{x = y = 0}} \bigg] \nonumber \\ 
	 &+ O(T/L). 
  \end{align*} 
By using similar arguments for the holomorphy of the error terms as in the Section 3.1, we end the proof of Lemma \ref{lemmac12}.
\subsection{Proof of Lemma \ref{lemmac22}}
This is the hardest case. Once again, we insert the definitions of the Feng mollifiers $\psi_2$ in the mean value integral $I_{22}(\alpha,\beta)$ so that
\begin{align}
  {I_{22}}(\alpha ,\beta ) &= \int_{ - \infty }^\infty  {w(t)\zeta (\tfrac{1}{2} + \alpha  + it)\zeta (\tfrac{1}{2} + \beta  - it){\psi _F}\overline {{\psi _F}} ({\sigma _0} + it)dt}  \nonumber \\
   &= \int_{ - \infty }^\infty  {w(t)\zeta (\tfrac{1}{2} + \alpha  + it)\zeta (\tfrac{1}{2} + \beta  - it)} \sum\limits_{{h_1} \leqslant {y_2}} {\frac{{\mu ({h_1})}}{{h_1^{1/2 + it}}}} \nonumber \\
	 &\quad \times \sum\limits_{{\ell _1} = 2}^{{K}} {\sum\limits_{{p_1}{p_2} \cdots {p_{{\ell _1}}}|{h_1}} {\frac{{\log {p_1}\log {p_2} \cdots \log {p_{{\ell_1}}}}}{{{{\log }^{{\ell _1}}}{y_2}}}} } {P_{{\ell _1}}}[h_1] \nonumber \\
   &\quad\times \sum\limits_{{h_2} \leqslant {y_2}} {\frac{{\mu ({h_2})}}{{h_2^{1/2 - it}}}} \sum\limits_{{\ell _2} = 2}^{{K}} {\sum\limits_{{q_1}{q_2} \cdots {q_{\ell_2}}|{h_2}} {\frac{{\log {q_1}\log {q_2} \cdots \log {q_{{\ell _2}}}}}{{{{\log }^{{\ell _2}}}{y_2}}}} } {P_{{\ell _2}}}[h_2]dt \nonumber \\
   &= \sum\limits_{{h_1},{h_2} \leqslant {y_2}} {\frac{{\mu ({h_1})\mu ({h_2})}}{{\sqrt {{h_1}{h_2}} }}} \nonumber \\
   &\quad \times \sum\limits_{{\ell _1} = 2}^{{K}} {\sum\limits_{{\ell _2} = 2}^{{K}} {\sum\limits_{{p_1}{p_2} \cdots {p_{{\ell _1}}}|{h_1}} {\sum\limits_{{q_1}{q_2} \cdots {q_{\ell_2}}|{h_2}} {\frac{{\log {p_1}\log {p_2} \cdots \log {p_{{\ell _1}}}\log {q_1}\log {q_2} \cdots \log {q_{{\ell _2}}}}}{{{{\log }^{{\ell _1} + {\ell _2}}}{y_2}}}} } } }  \nonumber \\
   &\quad \times P_{\ell_1}[h_1]P_{\ell_2}[h_2]\int_{ - \infty }^\infty  {w(t)\zeta (\tfrac{1}{2} + \alpha  + it)\zeta (\tfrac{1}{2} + \beta  - it){{\left( {\frac{{{h_1}}}{{{h_2}}}} \right)}^{ - it}}dt} . \nonumber 
\end{align}
We already explained in the computation of $I_{12}(\alpha,\beta)$ how to deal with this integral, namely write $I_{22}(\alpha,\beta)=I'_{22}(\alpha,\beta)+I''_{22}(\alpha,\beta)$, where $I''_{22}(\alpha,\beta)$ can be obtained from $I'_{22}$ by switching $\alpha$ and $-\beta$ and multiplying by
\[
\bigg( \frac{t}{2\pi} \bigg)^{-\alpha-\beta} = T^{-\alpha - \beta} + O(L^{-1}).
\]
We now use the Mellin integral representations of the polynomials
\[{P_{{\ell _1}}}[{h_1}] = \sum\limits_i {\frac{{{b_{i,{\ell _1}}}}}{{{{\log }^i}{y_2}}}} {(\log ({y_2}/{h_1}))^i} 
= 
\sum\limits_i {\frac{{{b_{i,{\ell _1}}}i!}}{{{{\log }^i}{y_2}}}} \frac{1}{{2\pi i}}\int_{(1)} {{{\left( {\frac{{{y_2}}}{{{h_1}}}} \right)}^{{u}}}\frac{{d{{u}}}}{{{u^{i + 1}}}}}, \]
and
\[{P_{{\ell _2}}}[{h_2}] = \sum\limits_j {\frac{{{b_{j,{\ell _2}}}}}{{{{\log }^j}{y_2}}}} {(\log ({y_2}/{h_2}))^j} 
= \sum\limits_j {\frac{{{b_{j,{\ell _2}}}j!}}{{{{\log }^j}{y_2}}}} \frac{1}{{2\pi i}}\int_{(1)} {{{\left( {\frac{{{y_2}}}{{{h_2}}}} \right)}^{{s}}}\frac{{d{{s}}}}{{{{{s}}^{j + 1}}}}} .\]
This leaves us with
\begin{align}
  {I'_{22}}(\alpha ,\beta ) &= \int_{ - \infty }^\infty  {w(t)} \sum\limits_{{\ell _1} = 2}^{{K}} {\sum\limits_{{\ell _2} = 2}^{{K}} {\sum\limits_{i,j} {\frac{{{b_{i,{\ell _1}}}i!}}{{{{\log }^{i + j}}{y_2}}}\frac{{{b_{j,{\ell _2}}}j!}}{{{{\log }^{{\ell _1} + {\ell _2}}}{y_2}}}} } }  \nonumber \\
   &\quad \times {\left( {\frac{1}{{2\pi i}}} \right)^3}\int_{(1)} \int_{(1)} \int_{(1)} y_2^{s + u} g_{\alpha,\beta}(z,t)\frac{{G(z)}}{z}   
   \sum\limits_{m{h_1} = n{h_2}} {\frac{{\mu ({h_1})\mu ({h_2})}}{{{h_1^{1/2  +{u} }}{h_2^{1/2 +{s} }}{m^{1/2 + \alpha + z}}{n^{1/2 + \beta + z}}}}}  \nonumber \\
   &\quad \times \sum\limits_{{p_1}{p_2} \cdots {p_{{\ell _1}}}|{h_1}} {\sum\limits_{{q_1}{q_2} \cdots {q_{\ell 2}}|{h_2}} {\log {p_1}\log {p_2} \cdots \log {p_{{\ell _1}}}\log {q_1}\log {q_2} \cdots \log {q_{{\ell _1}}}} } dz\frac{{du}}{{{u^{j + 1}}}}\frac{{ds}}{{{u^{i + 1}}}}dt. \nonumber 
\end{align}
We now have to compute the arithmetical sum $\sum_{mh_1=nh_2}$ with $p$-adic analysis. The first step is to consolidate the two sums over primes into a single sum. This is accomplished by the use of Lemma \ref{lemmacombinatorics}. Let us define
\begin{align}
  S_{{\ell _1},{\ell _2},k} &= \sum\limits_{m{h_1} = n{h_2}} {\frac{{\mu ({h_1})\mu ({h_2})}}{{h_1^{1/2 + u}h_2^{1/2 + s}{m^{1/2 + \alpha  + z}}{n^{1/2 + \beta  + z}}}}}  \nonumber \\
   &\quad \times \sum_{\substack{{p_1} \cdots {p_k}{q_1} \cdots {q_{{\ell _1} - k}}{r_1} \cdots {r_{{\ell _2} - k}}|{h_1}{h_2}/\gcd ({h_1},{h_2}) \\ {p_1} \cdots {p_k}|\gcd ({h_1},{h_2}) \\ {q_1} \cdots {q_{{\ell _1} - k}}|{h_1} \\ {r_1} \cdots {r_{{\ell _2} - k}}|{h_2}}} {{{\log }^2}{p_1} \cdots {{\log }^2}{p_k}\log {q_1} \cdots \log {q_{{\ell _1} - k}} \cdots \log {r_1} \cdots \log {r_{{\ell _2} - k}}} .  
\end{align}
The next step is to swap the order of the sums so that
\begin{align}
  {S_{{\ell _1},{\ell _2},{{k}}}} &= {( - 1)^{{\ell _1} + {\ell _2}}}
  \sum_{\substack{{p_i} \ne {p_j} \\ {q_i} \ne {q_j} \\ {r_i} \ne {r_j} \\ {p_i} \ne {q_i} \ne {r_i}}} {\frac{{{{\log }^2}{p_1} \cdots {{\log }^2}{p_k}\log {q_1} \cdots \log {q_{{\ell _1} - k}}\log {r_1} \cdots \log {r_{{\ell _2} - k}}}}{{{{({p_1} \cdots {p_k}{q_1} \cdots {q_{{\ell _1} - k}})}^{1/2 + {{u}}}}{{({p_1} \cdots {p_k}{r_1} \cdots {r_{{\ell _2} - k}})}^{1/2 + {{s}}}}}}}  \nonumber \\
   &\quad \times \sum_{\substack {m{{\tilde h}_1}{p_1} \cdots {p_k}{q_1} \cdots {q_{{\ell _1} - k}} = n \\ {{\tilde h}_2}{p_1} \cdots {p_k}{r_1} \cdots {r_{{\ell _2} - k}},({{\tilde h}_1},{p_1} \cdots {p_k}{q_1} \cdots {q_{{\ell _1} - k}}) = 1 \\ ({{\tilde h}_2},{p_1} \cdots {p_k}{r_1} \cdots {r_{{\ell _2} - k}}) = 1}} {\frac{{\mu ({{\tilde h}_1})\mu ({{\tilde h}_2})}}{{{{({{\tilde h}_1})}^{1/2 + {u}}}{{({{\tilde h}_2})}^{1/2 + {s}}}{m^{1/2 + \alpha  + z}}{n^{1/2 + \beta  + z}}}}},  \nonumber
\end{align}
by making the changes
\begin{align}
  {h_1} &= {{\tilde h}_1}{p_1} \cdots {p_k}{q_1} \cdots {q_{{\ell _1} - k}}, \nonumber \\
  {h_2} &= {{\tilde h}_2}{p_1} \cdots {p_k}{r_1} \cdots {r_{{\ell _2} - k}}, \nonumber  
\end{align} 
implying that
\begin{align}
  ({{\tilde h}_1},{p_1} \cdots {p_k}{q_1} \cdots {q_{{\ell _1} - k}}) &= 1, \nonumber \\
  ({{\tilde h}_2},{p_1} \cdots {p_k}{r_1} \cdots {r_{{\ell _2} - k}}) &= 1, \nonumber \\
  ({q_1} \cdots {q_{{\ell _1} - k}},{r_1} \cdots {r_{{\ell _2} - k}}) &= 1, \nonumber  
\end{align}
so that
\begin{align}
  {S_{{\ell _1},{\ell _2}, {{k}}}} &= {( - 1)^{{\ell _1} + {\ell _2}}}
  \sum_{\substack{{p_i} \ne {p_j} \\ {q_i} \ne {q_j} \\ {r_i} \ne {r_j} \\ {p_i} \ne {q_i} \ne {r_i}}} 
  {\frac{{{{\log }^2}{p_1} \cdots {{\log }^2}{p_k}\log {q_1} \cdots \log {q_{{\ell _1} - k}}\log {r_1} \cdots \log {r_{{\ell _2} - k}}}}{{{{({p_1} \cdots {p_k}{q_1} \cdots {q_{{\ell _1} - k}})}^{1/2 + {{u}}}}{{({p_1} \cdots {p_k}{r_1} \cdots {r_{{\ell _2} - k}})}^{1/2 + {{s}}}}}}}  \nonumber \\
   &\quad \times \sum_{\substack {m{{\tilde h}_1}{p_1} \cdots {p_k}{q_1} \cdots {q_{{\ell _1} - k}} = n \\ {{\tilde h}_2}{p_1} \cdots {p_k}{r_1} \cdots {r_{{\ell _2} - k}},({{\tilde h}_1},{p_1} \cdots {p_k}{q_1} \cdots {q_{{\ell _1} - k}}) = 1 \\ ({{\tilde h}_2},{p_1} \cdots {p_k}{r_1} \cdots {r_{{\ell _2} - k}}) = 1}} {\frac{{\mu ({{\tilde h}_1})\mu ({{\tilde h}_2})}}{{{{({{\tilde h}_1})}^{1/2 + {{u}}}}{{({{\tilde h}_2})}^{1/2 + {{s}}}}{m^{1/2 + \alpha  + z}}{n^{1/2 + \beta  + z}}}}}  .\nonumber  
\end{align}
Here the $p$'s, the $q$'s and the $r$'s are all distinct primes. Let us define the inner sum to be $\tilde{S}_{\ell_1,\ell_2,k} $ and let us recall that $\nu_p(n)=n'$ is the number of times the prime $p$ appears in $n$ so that
\begin{align}
  {{\tilde S}_{{\ell _1},{\ell _2},{{k}}}} 
  &= 
  \sum_{\substack {m{{\tilde h}_1}{p_1} \cdots {p_k}{q_1} \cdots {q_{{\ell _1} - k}} = n \\ {{\tilde h}_2}{p_1} \cdots {p_k}{r_1} \cdots {r_{{\ell _2} - k}},({{\tilde h}_1},{p_1} \cdots {p_k}{q_1} \cdots {q_{{\ell _1} - k}}) = 1 \\ ({{\tilde h}_2},{p_1} \cdots {p_k}{r_1} \cdots {r_{{\ell _2} - k}}) = 1}} {\frac{{\mu ({{\tilde h}_1})\mu ({{\tilde h}_2})}}{{{{({{\tilde h}_1})}^{1/2 + {{u}}}}{{({{\tilde h}_2})}^{1/2 + {{s}}}}{m^{1/2 + \alpha  + z}}{n^{1/2 + \beta  + z}}}}}  \nonumber \\
   &= 
   \prod\limits_{p \in \{ {p_1}, \cdots ,{p_k}\} } {\sum\limits_{n' = m'} {\frac{1}{{{{({p^{m'}})}^{1/2 + \alpha  + z}}{{({p^{n'}})}^{1/2 + \beta  + z}}}}} }  \nonumber \\
   &\quad \times \prod\limits_{q \in \{ {q_1}, \cdots ,{q_{{\ell _1} - k}}\} } {\sum\limits_{1+m' = n' + {{\tilde h}_2}'} {\frac{{\mu ({q^{{{\tilde h}_2}'}})}}{{{{({q^{{{\tilde h}_2}'}})}^{1/2 + {{s}}}}{{({q^{m'}})}^{1/2 + \alpha  + z}}{{({q^{n'}})}^{1/2 + \beta  + z}}}}} }  \nonumber \\
   &\quad \times \prod\limits_{r \in \{ {r_1}, \cdots ,{r_{{\ell _2} - k}}\} } {\sum\limits_{{{\tilde h}_1}' + m' = n'+1} {\frac{{\mu ({r^{{{\tilde h}_1}'}})}}{{{{({r^{{{\tilde h}_1}'}})}^{1/2 + {{u}}}}{{({r^{m'}})}^{1/2 + \alpha  + z}}{{({r^{n'}})}^{1/2 + \beta  + z}}}}} }  \nonumber \\
   &\quad \times \prod\limits_{p \notin \{ {p_1}, \cdots ,{p_k}\}  \cup \{ {q_1}, \cdots ,{q_{{\ell _1} - k}}\}  \cup \{ {r_1}, \cdots ,{r_{{\ell _2} - k}}\} } \nonumber \\
	 &\quad \times \sum\limits_{{{\tilde h}_1}' + m' = n' + {{\tilde h}_2}'} {\frac{{\mu ({p^{{{\tilde h}_1}'}})\mu ({p^{{{\tilde h}_2}'}})}}{{{{({p^{{{\tilde h}_1}'}})}^{1/2 + {{u}}}}{{({p^{{{\tilde h}_2}'}})}^{1/2 + {{s}}}}{{({p^{m'}})}^{1/2 + \alpha  + z}}{{({p^{n'}})}^{1/2 + \beta  + z}}}}}   \nonumber \\
   &= \frac{{{\Pi _1}(\alpha ,\beta ,s,u,z){\Pi _2}(\alpha ,\beta ,s,u,z){\Pi _3}(\alpha ,\beta ,s,u,z){\Pi _4}(\alpha ,\beta ,s,u,z)}}{{{\Pi _5}(\alpha ,\beta ,s,u,z)}}. \nonumber  
\end{align}
Each product is evaluated to
\begin{align}
  {\Pi _1}(\alpha ,\beta ,s,u,z) &= \prod\limits_p {\sum\limits_{{{\tilde h}_1}' + m' = n' + {{\tilde h}_2}'} {\frac{{\mu ({p^{{{\tilde h}_1}'}})\mu ({p^{{{\tilde h}_2}'}})}}{{{{({p^{{{\tilde h}_1}'}})}^{1/2 + {{u}}}}{{({p^{{{\tilde h}_2}'}})}^{1/2 + {{s}}}}{{({p^{m'}})}^{1/2 + \alpha  + z}}{{({p^{n'}})}^{1/2 + \beta  + z}}}}} }  \nonumber \\
   &= \prod\limits_p {\left( {1 + \frac{1}{{{p^{1 + s + u}}}} - \frac{1}{{{p^{1 + s + \alpha  + z}}}} - \frac{1}{{{p^{1 + u + \beta  + z}}}} + \frac{1}{{{p^{1 + \alpha  + \beta  + 2z}}}} + O({p^{ - 2 + \varepsilon }})} \right)},  \nonumber 
\end{align}
then
\begin{align}
{\Pi _2}(\alpha ,\beta ,s,u,z) &= \prod\limits_{p \in \{ {p_1}, \cdots ,{p_k}\} } {\sum\limits_{n' = m'} {\frac{1}{{{{({p^{m'}})}^{1/2 + \alpha  + z}}{{({p^{n'}})}^{1/2 + \beta  + z}}}}} } \nonumber \\
 &= \prod\limits_{p \in \{ {p_1}, \cdots ,{p_k}\} } {\left( {1 + \frac{1}{{{p^{1 + \alpha  + \beta  + 2z}}}} + O({p^{ - 2 + \varepsilon }})} \right)}, \end{align}
followed by
\begin{align}
  {\Pi _3}(\alpha ,\beta ,s,u,z) &= \prod\limits_{q \in \{ {q_1}, \cdots ,{q_{{\ell _1} - k}}\} } {\sum\limits_{m' + 1 = n' + {{\tilde h}_2}'} {\frac{{\mu ({q^{{{\tilde h}_2}'}})}}{{{{({q^{{{\tilde h}_2}'}})}^{1/2 + {{s}}}}{{({q^{m'}})}^{1/2 + \alpha  + z}}{{({q^{n'}})}^{1/2 + \beta  + z}}}}} }  \nonumber \\
   &= \prod\limits_{q \in \{ {q_1}, \cdots ,{q_{{\ell _1} - k}}\} } {\left( {- \frac{1}{{{q^{1/2+{{s}}}}}} + \frac{1}{{{q^{1/2 + \beta  + z }}}} + O({q^{ - 2 + \varepsilon }})} \right)},  \nonumber  
\end{align}
as well as
\begin{align}
  {\Pi _4}(\alpha ,\beta ,s,u,z) &= \prod\limits_{r \in \{ {r_1}, \cdots ,{r_{{\ell _2} - k}}\} } {\sum\limits_{{{\tilde h}_1}' + m' = n'+1} {\frac{{\mu ({r^{{{\tilde h}_1}'}})}}{{{{({r^{{{\tilde h}_1}'}})}^{1/2 + {{u}}}}{{({r^{m'}})}^{1/2 + \alpha  + z}}{{({r^{n'}})}^{1/2 + \beta  + z}}}}} }  \nonumber \\
   &= \prod\limits_{r \in \{ {r_1}, \cdots ,{r_{{\ell _2} - k}}\} } {\left( {- \frac{1}{{{r^{1/2+{{u}}}}}} + \frac{1}{{{r^{1/2 + \alpha  + z }}}} + O({r^{ - 2 + \varepsilon }})} \right)} , \nonumber 
\end{align}
and finally
\begin{align}
  {\Pi _5}(\alpha ,\beta ,s,u,z) &= \prod\limits_{p \in \{ {p_1}, \cdots ,{p_k}\}  \cup \{ {q_1}, \cdots ,{q_{{\ell _1} - k}}\}  \cup \{ {r_1}, \cdots ,{r_{{\ell _2} - k}}\} } \nonumber \\
	 &\quad \times \sum\limits_{{{\tilde h}_1}' + m' = n' + {{\tilde h}_2}'} {\frac{{\mu ({p^{{{\tilde h}_1}'}})\mu ({p^{{{\tilde h}_2}'}})}}{{{{({p^{{{\tilde h}_1}'}})}^{1/2 + {{u}}}}{{({p^{{{\tilde h}_2}'}})}^{1/2 + {{s}}}}{{({p^{m'}})}^{1/2 + \alpha  + z}}{{({p^{n'}})}^{1/2 + \beta  + z}}}}}   \nonumber \\
   &= \prod\limits_{p \in \{ {p_1}, \cdots ,{p_k}\}  \cup \{ {q_1}, \cdots ,{q_{{\ell _1} - k}}\}  \cup \{ {r_1}, \cdots ,{r_{{\ell _2} - k}}\} } \nonumber \\
	 &\quad \times \left( {1 + \frac{1}{{{p^{1 + s + u}}}} - \frac{1}{{{p^{1 + s + \alpha  + z}}}} - \frac{1}{{{p^{1 + u + \beta  + z}}}} + \frac{1}{{{p^{1 + \alpha  + \beta  + 2z}}}} + O({p^{ - 2 + \varepsilon }})} \right) . \nonumber  
\end{align}
This leaves us with
\begin{align}
  {{\tilde S}_{{\ell _1},{\ell _2},{{k}}}} &= \prod\limits_p {\left( {1 + \frac{1}{{{p^{1 + s + u}}}} - \frac{1}{{{p^{1 + s + \alpha  + z}}}} - \frac{1}{{{p^{1 + u + \beta  + z}}}} + \frac{1}{{{p^{1 + \alpha  + \beta  + 2z}}}} + O({p^{ - 2 + \varepsilon }})} \right)}  \nonumber \\
   &= \frac{{\zeta (1 + s + u)\zeta (1 + \alpha  + \beta  + 2z)}}{{\zeta (1 + s + \alpha  + z)\zeta (1 + u + \beta  + z)}}{A_{\alpha ,\beta }}(s,u,z), \nonumber 
\end{align}
where $A$ is an arithmetical factor that is given by an absolutely convergent Euler product in some product of half-planes containing the origin. From our previous analysis of the $I_{12}(\alpha,\beta)$ case, we know that $A_{0,0}(z,z,z)=1$ for all values of $z$. Therefore we end up with
\begin{align}
  S_{{\ell _1},{\ell _2},k} &= \frac{{\zeta (1 + s + u)\zeta (1 + \alpha  + \beta  + 2z)}}{{\zeta (1 + s + \alpha  + z)\zeta (1 + u + \beta  + z)}}{A_{\alpha ,\beta }}(s,u,z) \nonumber \\
   &\quad \times {( - 1)^{{\ell _1} + {\ell _2}}}\sum_{\substack {{p_i} \ne {p_j} \\ {q_i} \ne {q_j} \\ {r_i} \ne {r_j} \\ {p_i} \ne {q_i} \ne {r_i}}} {{{\log }^2}{p_1} \cdots {{\log }^2}{p_k}\log {q_1} \cdots \log {q_{{\ell _1} - k}}\log {r_1} \cdots \log {r_{{\ell _2} - k}}}  \nonumber \\
   &\quad \times \prod\limits_{p \in \{ {p_1}, \cdots ,{p_k}\} } {\frac{{{E_1}(p) + O({p^{ - 2 + \varepsilon }})}}{{1 + \tfrac{1}{{{p^{1 + s + \alpha  + z}}}} + \tfrac{1}{{{p^{1 + u + \beta  + z}}}} - \tfrac{1}{{{p^{1 + \alpha  + \beta  + 2z}}}} + {E_1}(p) + O({p^{ - 2 + \varepsilon }})}}}  \nonumber \\
   &\quad \times \prod\limits_{q \in \{ {q_1}, \cdots ,{q_{{\ell _1} - k}}\} } {\frac{{{E_2}(q) + O({q^{ - 2 + \varepsilon }})}}{{1 + \tfrac{1}{{{q^{1 + s + \alpha  + z}}}} - \tfrac{1}{{{q^{1 + \alpha  + \beta  + 2z}}}} -E_2(q) + O({q^{ - 2 + \varepsilon }})}}}  \nonumber \\
   &\quad \times \prod\limits_{r \in \{ {r_1}, \cdots ,{r_{{\ell _2} - k}}\} } {\frac{{{E_3}(r) + O({r^{ - 2 + \varepsilon }})}}{{1 + \tfrac{1}{{{r^{1 + u + \beta  + z}}}} - \tfrac{1}{{{r^{1 + \alpha  + \beta  + 2z}}}} -E_3(r) + O({r^{ - 2 + \varepsilon }})}}},  \nonumber 
\end{align}
where
\[{E_1}(p) = \frac{1}{{{p^{1 + s + u}}}},\]
and
\[{E_2}(q) = \frac{1}{{{q^{1/2 + u}}}}\left( {\frac{1}{{{q^{1/2 + s}}}} - \frac{1}{{{q^{1/2 + \beta  + z}}}}} \right) = \frac{1}{{{q^{1 + s + u}}}} - \frac{1}{{{q^{1 + \beta  + u + z}}}},\]
and finally
\[{E_3}(r) = \frac{1}{{{r^{1/2 + s}}}}\left( {\frac{1}{{{q^{1/2 + u}}}} - \frac{1}{{{q^{1/2 + \alpha  + z}}}}} \right) = \frac{1}{{{r^{1 + s + u}}}} - \frac{1}{{{r^{1 + \alpha  + s + z}}}}.\]
We define $H_{\ell_1,\ell_2,k}$ to be the last part of $S_{\ell_1,\ell_2,k}$. This means that
\begin{align}
  H_{\ell_1, \ell_2, k} &= {( - 1)^{{\ell _1} + {\ell _2}}}\sum_{\substack{{p_i} \ne {p_j} \\ {q_i} \ne {q_j} \\ {r_i} \ne {r_j} \\ {p_i} \ne {q_i} \ne {r_i}}}  ({E_1}(p) + O({p^{ - 2 + \varepsilon }})){\log ^2}p \nonumber \\
	 &\quad \times \left( {1 - {E_1}(p) - \frac{1}{{{p^{1 + s + \alpha  + z}}}} - \frac{1}{{{p^{1 + u + \beta  + z}}}} + \frac{1}{{{p^{1 + \alpha  + \beta  + 2z}}}} + O({p^{ - 2 + \varepsilon }})} \right) \nonumber \\
   &\quad \times ({E_2}(q) + O({q^{ - 2 + \varepsilon }}))\log q\left( {1 + {E_2}(q) - \frac{1}{{{q^{1 + s + \alpha  + z}}}} + \frac{1}{{{q^{1 + \alpha  + \beta  + 2z}}}} + O({q^{ - 2 + \varepsilon }})} \right) \nonumber \\
   &\quad \times ({E_3}(r) + O({r^{ - 2 + \varepsilon }}))\log r\left( {1 + {E_3}(r) - \frac{1}{{{r^{1 + u + \beta  + z}}}} + \frac{1}{{{r^{1 + \alpha  + \beta  + 2z}}}} + O({r^{ - 2 + \varepsilon }})} \right) \nonumber \\
   &= {( - 1)^{{\ell _1} + {\ell _2}}}\sum_{\substack{{p_i} \ne {p_j} \\ {q_i} \ne {q_j} \\ {r_i} \ne {r_j} \\ {p_i} \ne {q_i} \ne {r_i}}}   \prod\limits_{p \in \{ {p_1}, \cdots ,{p_k}\} } {\left( {{E_1}(p){{\log }^2}p + O\left( {\frac{{{{\log }^2}p}}{{{p^{2 - \varepsilon }}}}} \right)} \right)}  \nonumber \\
   &\quad \times \prod\limits_{q \in \{ {q_1}, \cdots ,{q_{{\ell _1} - k}}\} } {\left( {{E_2}(q)\log q + O\left( {\frac{{\log q}}{{{q^{2 - \varepsilon }}}}} \right)} \right)}  \prod\limits_{r \in \{ {r_1}, \cdots ,{r_{{\ell _2} - k}}\} } {\left( {{E_3}(r)\log r + O\left( {\frac{{\log r}}{{{r^{2 - \varepsilon }}}}} \right)} \right)}  \nonumber \\
  & + O(f(p^{-2+\varepsilon},q^{-2+\varepsilon},r^{-2+\varepsilon})) ,\nonumber
   \end{align}
for some polynomial $f$. Applying the inclusion-exclusion principle we then have
\begin{align}
  {H_{{\ell _1},{\ell _2},{{k}}}} &= {( - 1)^{{\ell _1} + {\ell _2}}}{\left( {\sum\limits_p {{E_1}(p){{\log }^2}p}  + O\left( {\frac{{{{\log }^2}p}}{{{p^{2 - \varepsilon }}}}} \right)} \right)^k} \nonumber \\
   &\quad \times {\left( {\sum\limits_q {{E_2}(q)\log q}  + O\left( {\frac{{\log q}}{{{q^{2 - \varepsilon }}}}} \right)} \right)^{{\ell _1} - k}}{\left( {\sum\limits_r {{E_3}(r)\log r}  + O\left( {\frac{{\log r}}{{{r^{2 - \varepsilon }}}}} \right)} \right)^{{\ell _2} - k}} \nonumber + \sum\limits_{p,q,r} {B(p,q,r)} , \nonumber 
\end{align}
where
\[
B(p,q,r) \ll _{\alpha ,\beta ,s,u,z,\varepsilon } f\bigg(\frac{1}{{{p^{2 - \varepsilon }}}},\frac{1}{q^{2-\varepsilon}},\frac{1}{r^{2-\varepsilon}}\bigg).
\]
As in the previous crossterm, we now need to identify the logarithms of the primes with the signature of the von Mangoldt functions $\Lambda(n)$ and $\Lambda_2(n)$. With this in mind, we first write
\[\frac{{\zeta '}}{\zeta }(s) =  - \sum\limits_{n = 1}^\infty  {\Lambda (n){n^{ - s}}}  =  - \sum\limits_p {\frac{{\log p}}{{{p^s}}}{{\left( {1 - \frac{1}{{{p^s}}}} \right)}^{ - 1}}}  
=  - \sum\limits_p {\frac{{\log p}}{p^s}}  + O\left( {\frac{{\log p}}{{{p^{2s}}}}} \right),\]
and
\[\frac{{\zeta ''}}{\zeta }(s) = \sum\limits_{n = 1}^\infty  {{\Lambda _2}(n){n^{ - s}}}  =  \sum\limits_p {\frac{{{{\log }^2}p}}{{{p^s}}}{{\left( {1 - \frac{1}{{{p^s}}}} \right)}^{ - 1}}}  
= \sum\limits_p {\frac{{{{\log }^2}p}}{p^s}}  + O\left( {\frac{{{{\log }^2}p}}{{{p^{2s}}}}} \right),\]
for $\real(s)>1$. This means that
\begin{align}
  {H_{{\ell _1},{\ell _2},{{k}}}} &= {( - 1)^{{\ell _1} + {\ell _2}}}{\left( {\frac{{\zeta ''}}{\zeta }(1 + s + u)} \right)^k}{\left( { - \frac{{\zeta '}}{\zeta }(1 + s + u) + \frac{{\zeta '}}{\zeta }(1 + \beta  + u + z)} \right)^{{\ell _1} - k}} \nonumber \\
   &\quad \times {\left( { - \frac{{\zeta '}}{\zeta }(1 + s + u) + \frac{{\zeta '}}{\zeta }(1 + \alpha  + s + z)} \right)^{{\ell _2} - k}} + D(\alpha ,\beta ,s,u,z) \nonumber \\
   &= {( - 1)^{{\ell _1} + {\ell _2}}} {( - {V_1})^k}{( - {V_2})^{{\ell _1} - k}}{( - {V_3})^{{\ell _2} - k}} \nonumber \\
	 & \quad + \sum\limits_{l = 0}^{k - 1} {V_1^l{A_l}(\alpha ,\beta ,s,u,z)} \sum\limits_{m = 0}^{{\ell _1} - k - 1} {V_2^m{B_m}(\alpha ,\beta ,s,u,z)} \sum\limits_{n = 0}^{{\ell _2} - k - 1} {V_3^n{C_n}(\alpha ,\beta ,s,u,z)}, \nonumber 
\end{align}
where $D(\alpha ,\beta ,s,u,z)$ are terms of smaller order and where
\[{V_1} =  - \frac{{\zeta ''}}{\zeta }(1 + s + u),\quad {V_2} = \frac{{\zeta '}}{\zeta }(1 + s + u) - \frac{{\zeta '}}{\zeta }(1 + \beta  + u + z),\quad {V_3} = \frac{{\zeta '}}{\zeta }(1 + s + u) - \frac{{\zeta '}}{\zeta }(1 + \alpha  + s + z).\]
Moreover, we also have that
\[
{A_l}(\alpha ,\beta ,s,u,z) \ll _{\alpha ,\beta ,s,u,z,\varepsilon } \sum\limits_p {\frac{{{{\log }^2}p}}{{{p^{2 - \varepsilon }}}}} ,\quad {B_m}(\alpha ,\beta ,s,u,z),{C_n}(\alpha ,\beta ,s,u,z) \ll _{\alpha ,\beta ,s,u,z,\varepsilon } \sum\limits_p {\frac{{\log p}}{{{p^{2 - \varepsilon }}}}} .\]
All of these terms are analytic in a larger region of the complex plane, thus we are only interested in the term ${( - {V_1})^k}{( - {V_2})^{{\ell _1} - k}}{( - {V_3})^{{\ell _2} - k}}$. Consequently, the end result of this computation is that
\begin{align}
  {I'_{22}}(\alpha ,\beta ) &= \int_{ - \infty }^\infty  w(t)\sum\limits_{{\ell _1} = 2}^K \sum\limits_{{\ell _2} = 2}^K \sum\limits_{i,j} \sum\limits_{k = 0}^{\min ({\ell _1},{\ell _2})} {\binom{\ell_1}{k}{{({\ell _2})}_k}}     \frac{{{b_{i,{\ell _1}}}i!}}{{{{\log }^{i + j}}{y_2}}}\frac{{{b_{j,{\ell _2}}}j!}}{{{{\log }^{{\ell _1} + {\ell _2}}}{y_2}}}  \nonumber \\
   &\quad \times {\left( {\frac{1}{{2\pi i}}} \right)^3}\int_{(1)} \int_{(1)} \int_{(1)}    y_2^{s+u} g_{\alpha,\beta}(z,t)\frac{{G(z)}}{z}\frac{{\zeta (1 + s + u)\zeta (1 + \alpha  + \beta  + 2z)}}{{\zeta (1 + s + \alpha  + z)\zeta (1 + u + \beta  + z)}}{A_{\alpha ,\beta }}(s,u,z) \nonumber \\
   &\quad \times {( - 1)^{{\ell _1} + {\ell _2}}}{\left( {\frac{{\zeta ''}}{\zeta }(1 + s + u)} \right)^k}{\left( { - \frac{{\zeta '}}{\zeta }(1 + s + u) + \frac{{\zeta '}}{\zeta }(1 + \beta  + u + z)} \right)^{{\ell _1} - k}} \nonumber \\
   &\quad \times {\left( { - \frac{{\zeta '}}{\zeta }(1 + s + u) + \frac{{\zeta '}}{\zeta }{{(1 + \alpha  + s + z)}}} \right)^{{\ell _2} - k}}dz\frac{{d{{u}}}}{{{u^{i + 1}}}}\frac{{d{{s}}}}{{{{{s}}^{j + 1}}}}dt. \nonumber  
\end{align}
As in the calculation of $I'_{12}$, we now take the $s,u,z$ contours of integration to $\delta>0$ small and fixed with $\delta < \varepsilon$, and then move $z$ to $-\delta + \varepsilon$, crossing a simple pole at $z=0$ only (since, yet again, $G(z)$ vanishes at the pole of $\zeta(1+\alpha+\beta+2z)$). 
The new line of integration with respect to $z$ contributes
\[
\ll T^{1+\varepsilon} \bigg( \frac{y_2^2}{T} \bigg)^{\delta} \ll T^{1-\varepsilon},
\]
since $\theta_2 = 1/2 - \varepsilon$. Write $I'_{22}(\alpha,\beta)=I'_{220}(\alpha,\beta) + O(T^{1-\varepsilon})$, where $I'_{220}(\alpha,\beta)$ corresponds to the residue at $z=0$, i.e.
\begin{align}
  I{'_{220}}(\alpha ,\beta ) &= \int_{ - \infty }^\infty  {w(t)\sum\limits_{{\ell _1} = 2}^K {\sum\limits_{{\ell _2} = 2}^K {\sum\limits_{i,j} {\sum\limits_{k = 0}^{\min ({\ell _1},{\ell _2})} {{{\binom{\ell _1}{k}}}{{({\ell _2})}_k}} } } } } \frac{{{b_{i,{\ell _1}}}i!}}{{{{\log }^{i + j}}{y_2}}}\frac{{{b_{j,{\ell _2}}}j!}}{{{{\log }^{{\ell _1} + {\ell _2}}}{y_2}}} \nonumber \\
   &\quad \times {\left( {\frac{1}{{2\pi i}}} \right)^2}\int_{(\delta )} {\int_{(\delta )} {\mathop {\operatorname{res} }\limits_{{{z}} = 0} }  \frac{{G(z)}}{z}} g_{\alpha,\beta}(z,t) y_2^{s+u}\frac{{\zeta (1 + s + u)\zeta (1 + \alpha  + \beta  + 2z)}}{{\zeta (1 + s + \alpha  + z)\zeta (1 + u + \beta  + z)}}{A_{\alpha ,\beta }}(s,u,z) \nonumber \\
   &\quad \times {( - 1)^{{\ell _1} + {\ell _2}}}{\left( {\frac{{\zeta ''}}{\zeta }(1 + s + u)} \right)^k}{\left( { - \frac{{\zeta '}}{\zeta }(1 + s + u) + \frac{{\zeta '}}{\zeta }(1 + \beta  + u + z)} \right)^{{\ell _1} - k}} \nonumber \\
   &\quad \times {\left( { - \frac{{\zeta '}}{\zeta }(1 + s + u) + \frac{{\zeta '}}{\zeta }{{(1 + \alpha  + s + z)}}} \right)^{{\ell _2} - k}}\frac{{d{{u}}}}{{{u^{i + 1}}}}\frac{{d{{s}}}}{{{{{s}}^{j + 1}}}}dt \nonumber \\
   &= \widehat {w}(0)\zeta (1 + \alpha  + \beta )\sum\limits_{{\ell _1} = 2}^K {\sum\limits_{{\ell _2} = 2}^K {\sum\limits_{i,j} {\sum\limits_{k = 0}^{\min ({\ell _1},{\ell _2})} {\binom{\ell_1}{k}{{({\ell _2})}_k}} } } {( - 1)^{{\ell _1} + {\ell _2}}}} \frac{{{b_{i,{\ell _1}}}i!}}{{{{\log }^{i + j}}{y_2}}}\frac{{{b_{j,{\ell _2}}}j!}}{{{{\log }^{{\ell _1} + {\ell _2}}}{y_2}}}{J_{22}}, \nonumber 
\end{align}
where
\begin{align}
  {J_{22}} &= {\left( {\frac{1}{{2\pi i}}} \right)^2}\int_{(\delta )} {\int_{(\delta )} {y_2^{s+u}\frac{{\zeta (1 + s + u)}}{{\zeta (1 + s + \alpha )\zeta (1 + u + \beta )}}{A_{\alpha ,\beta }}(s,u,0)} }  \nonumber \\
   &\quad \times {\left( {\frac{{\zeta ''}}{\zeta }(1 + s + u)} \right)^k}{\left( { - \frac{{\zeta '}}{\zeta }(1 + s + u) + \frac{{\zeta '}}{\zeta }(1 + \beta  + u)} \right)^{{\ell _1} - k}} \nonumber \\
   &\quad \times {\left( { - \frac{{\zeta '}}{\zeta }(1 + s + u) + \frac{{\zeta '}}{\zeta }{{(1 + \alpha  + s)}}} \right)^{{\ell _2} - k}}\frac{{d{{u}}}}{{{u^{i + 1}}}}\frac{{d{{s}}}}{{{{{s}}^{j + 1}}}} .\nonumber 
\end{align}
The next step is to employ the binomial theorem in the part of the integrand that involves $\zeta$ functions. Calling this part $\mathcal{Z}$, we then have
\begin{align}
  \mathcal{Z}(s,u) &:= \frac{{\zeta (1 + s + u)}}{{\zeta (1 + s + \alpha )\zeta (1 + u + \beta )}}{\left( {\frac{{\zeta ''}}{\zeta }(1 + s + u)} \right)^k} \nonumber \\
   &\quad \times {\left( { - \frac{{\zeta '}}{\zeta }(1 + s + u) + \frac{{\zeta '}}{\zeta }(1 + \beta  + u)} \right)^{{\ell _1} - k}}{\left( { - \frac{{\zeta '}}{\zeta }(1 + s + u) + \frac{{\zeta '}}{\zeta }(1 + \alpha  + s)} \right)^{{\ell _2} - k}} \nonumber \\
   &= \frac{{\zeta (1 + s + u)}}{{\zeta (1 + s + \alpha )\zeta (1 + u + \beta )}}{\left( {\frac{{\zeta ''}}{\zeta }(1 + s + u)} \right)^k} \nonumber \\
   &\quad \times \sum\limits_{{r_1} = 0}^{{\ell _1} - k} {\binom{\ell _1-k}{r_1}{{\left( {\frac{{\zeta '}}{\zeta }(1 + \beta  + u)} \right)}^{{\ell _1} - k - {r_1}}}{{\left( { - \frac{{\zeta '}}{\zeta }(1 + s + u)} \right)}^{{r_1}}}}  \nonumber \\
   &\quad \times \sum\limits_{{r_2} = 0}^{{\ell _2} - k} {\binom{\ell _2-k}{r_2}{{\left( {\frac{{\zeta '}}{\zeta }(1 + \alpha  + s)} \right)}^{{\ell _2} - k - {r_2}}}{{\left( { - \frac{{\zeta '}}{\zeta }(1 + s + u)} \right)}^{{r_2}}}}  \nonumber \\
   &= \sum\limits_{{r_1} = 0}^{{\ell _1} - k} {\sum\limits_{{r_2} = 0}^{{\ell _2} - k} {\binom{\ell _1-k}{r_1}\binom{\ell _2-k}{r_2}} } \sum\limits_{n = 1}^\infty  {\frac{{(\mathbf{1} * \Lambda _2^{ * k} * {\Lambda ^{ * {r_1} + {r_2}}})(n)}}{{{n^{1 + s + u}}}}}  \nonumber \\
   &\quad \times \frac{1}{{\zeta (1 + \beta + u )}}{\left( {\frac{{\zeta '}}{\zeta }(1 + \beta  + u)} \right)^{{\ell _1} - k - {r_1}}}\frac{1}{{\zeta (1 + s + \alpha )}}{\left( {\frac{{\zeta '}}{\zeta }(1 + \alpha  + s)} \right)^{{\ell _2} - k - {r_2}}}, \nonumber  
\end{align}
where we have used the Dirichlet convolution of
\[\zeta (s) = \sum\limits_{n = 1}^\infty  {\frac{1}{{{n^s}}}} ,\quad \frac{{\zeta '}}{\zeta }(s) =  - \sum\limits_{n = 1}^\infty  {\frac{{\Lambda (n)}}{{{n^s}}}} ,\quad \textnormal{and} \quad \frac{{\zeta ''}}{\zeta }(s) = \sum\limits_{n = 1}^\infty  {\frac{{{\Lambda _2}(n)}}{{{n^s}}}}, \]
for $\real(s)>1$. Now we take $\delta \asymp L^{-1}$ and bound the integral trivially to get $J_{22} \ll L^{i+j-1}$. This means that we can use a Taylor series expansion so that $A_{\alpha,\beta}(s,u,0)=A_{0,0}(0,0,0)+O(|s|+|u|)$ to write $J_{22}(\alpha,\beta) = J'_{22}(\alpha,\beta) +O(L^{i+j-2})$, say. We recall that earlier we proved that $A_{0,0}(z,z,z)=1$ for all $z$, and hence $A_{0,0}(0,0,0)=1$. This has the effect of separating the complex variables $s$ and $u$ as follows
\[{J'_{22}} = \sum\limits_{n \leqslant {y_2}} {\sum\limits_{{r_1} = 0}^{{\ell _1} - k} {\sum\limits_{{r_2} = 0}^{{\ell _2} - k} {\binom{\ell _1-k}{r_1}\binom{\ell _2-k}{r_2}} } } \frac{{(\mathbf{1} * \Lambda _2^{ * k} * {\Lambda ^{ * {r_1} + {r_2}}})(n)}}{n}{L_{22,1}}{L_{22,2}},\]
where
\begin{align} \label{defL22}
{L_{22,1}} = \frac{1}{{2\pi i}}\int_{(\delta )} {{{\left( {\frac{{{y_2}}}{n}} \right)}^s}\frac{1}{{\zeta (1 + s + \alpha )}}{{\left( {\frac{{\zeta '}}{\zeta }(1 + \alpha  + s)} \right)}^{{\ell _2} - k - {r_2} }}\frac{{ds}}{{{{{s}}^{j + 1}}}}}, 
\end{align}
and
\[{L_{22,2}} = \frac{1}{{2\pi i}}\int_{(\delta )} {{{\left( {\frac{{{y_2}}}{n}} \right)}^u}\frac{1}{{\zeta (1 + \beta  + u)}}{{\left( {\frac{{\zeta '}}{\zeta }(1 + \beta  + u)} \right)}^{{\ell _1} - k - {r_1} }}\frac{{du}}{{{u^{{{i}} + 1}}}}} .\]
These two integrals are identical, up to the symmetries in $s/u$, $\ell_1 / \ell_2$, $\alpha / \beta$ and $r_1 / r_2$ and they were in fact treated in the $I_{12}(\alpha,\beta)$ case. The end results for the main terms are 
\[{L_{22,1}} = \frac{1}{2 \pi i} \oint \bigg( \frac{y_2}{n} \bigg)^s (s + \alpha)^{1-\ell_2  + k + r_2} \frac{ds}{s^{j+1}} = \frac{{{{( - 1)}^{{\ell _2} - k + {r_2}}}}}{{{{j}}!}}\frac{{{d^{1 - {\ell _2} + k + {r_2}}}}}{{d{x^{1 - {\ell _2} + k + {r_2}}}}}{e^{\alpha x}}{\left( {x + \log \frac{{{y_2}}}{n}} \right)^{{j}}}{\bigg|_{x = 0}},\]
and
\[{L_{22,2}} = \frac{1}{2 \pi i} \oint \bigg( \frac{y_2}{n} \bigg)^u (u + \beta)^{1-\ell_1  + k + r_1} \frac{du}{u^{i+1}} = \frac{{{{( - 1)}^{{\ell _1} - k + {r_1}}}}}{{{{i}}!}}\frac{{{d^{1 - {\ell _1} + k + {r_1}}}}}{{d{y^{1 - {\ell _1} + k + {r_1}}}}}{e^{\beta y}}{\left( {y + \log \frac{{{y_2}}}{n}} \right)^{{i}}}{\bigg|_{y = 0}}.\]
Next, we insert these results into $J_{22}$ and we end up with
\begin{align}
  {J'_{22}} &= \frac{1}{{i!}}\frac{1}{{j!}}\sum\limits_{n \leqslant {y_2}} {\sum\limits_{{r_1} = 0}^{{\ell _1} - k} {\sum\limits_{{r_2} = 0}^{{\ell _2} - k} {{{( - 1)}^{{\ell _1} + {\ell _2} - 2k + {r_1} + {r_2}}}\binom{\ell_1 - k}{r_1}\binom{\ell_2 - k}{r_2}} } } \frac{{(\mathbf{1} * \Lambda _2^{ * k} * {\Lambda ^{ * {r_1} + {r_2}}})(n)}}{n} \nonumber \\
   &\quad \times \frac{{{d^{1 - {\ell _2} + k + {r_2}}}}}{{d{x^{1 - {\ell _2} + k +- {r_2}}}}}\frac{{{d^{1 - {\ell _1} + k + {r_1}}}}}{{d{y^{1 - {\ell _1} + k + {r_1}}}}}{e^{\alpha x + \beta y}}{\left( {x + \log \frac{{{y_2}}}{n}} \right)^{{j}}}{\bigg|_{x = 0}}{\left( {y + \log \frac{{{y_2}}}{n}} \right)^{{i}}}{\bigg|_{y = 0}} + O(L^{i+j-2}). \nonumber  
\end{align}
To make matters easier, we again employ the change of variables
\[
x \to \frac{x}{\log y_2} \quad \textnormal{and} \quad y \to \frac{y}{\log y_2},
\]
and this produces
\begin{align}
  J'_{22} &= \frac{{{{\log }^{i + j - 2}}{y_2}}}{{i!j!}}\sum\limits_{n \leqslant {y_2}} {\sum\limits_{{r_1} = 0}^{{\ell _1} - k} {\sum\limits_{{r_2} = 0}^{{\ell _2} - k} {{{( - 1)}^{{\ell _1} + {\ell _2} - 2k + {r_1} + {r_2}}}\binom{\ell_1 - k}{r_1}\binom{\ell_2 - k}{r_2}} } } \frac{{(\mathbf{1} * \Lambda _2^{ * k} * {\Lambda ^{ * {r_1} + {r_2}}})(n)}}{n} \nonumber \\
   &\quad \times \frac{{{d^{1 - {\ell _2} + k + {r_2}}}}}{{d{x^{1 - {\ell _2} + k + {r_2}}}}}\frac{{{d^{1 - {\ell _1} + k + {r_1}}}}}{{d{y^{1 - {\ell _1} + k + {r_1}}}}}y_2^{\alpha x + \beta y}{\left( {x + \frac{{\log ({y_2}/n)}}{{\log {y_2}}}} \right)^{{j}}}{\bigg|_{x = 0}}{\left( {y + \frac{{\log ({y_2}/n)}}{{\log {y_2}}}} \right)^{{i}}}{\bigg|_{y = 0}} + O(L^{i+j-2}). \nonumber  
\end{align}
We are now ready to insert this into $I'_{220}$ so that
\begin{align}
  I'_{220}(\alpha ,\beta ) &= \frac{{\widehat {w}(0)}}{{(\alpha  + \beta ){{\log }^2}{y_2}}}\frac{{{d^2}}}{{dxdy}} \bigg[y_2^{\alpha x + \beta y}\sum\limits_{{\ell _1} = 2}^K {\sum\limits_{{\ell _2} = 2}^K {\frac{1}{{{{\log }^{{\ell _1} + {\ell _2}}}{y_2}}}\sum\limits_{i,j} {\sum\limits_{k = 0}^{\min ({\ell _1},{\ell _2})} {\binom{\ell_1}{k}{{({\ell _2})}_k}} } } } {a_{i,{\ell _1}}}{a_{j,{\ell _2}}} \nonumber \\
   &\quad \times \sum\limits_{n \leqslant {y_2}} {\sum\limits_{{r_1} = 0}^{{\ell _1} - k} {\sum\limits_{{r_2} = 0}^{{\ell _2} - k} {{{( - 1)}^{  {r_1} + {r_2}}}\binom{\ell_1 - k}{r_1}\binom{\ell_2 - k}{r_2}} } } \frac{{(\mathbf{1} * \Lambda _2^{ * k} * {\Lambda ^{ * {r_1} + {r_2}}})(n)}}{n} \nonumber \\
   &\quad \times \frac{{{d^{k - {\ell _2} + {r_2}}}}}{{d{x^{k - {\ell _2} + {r_2}}}}}\frac{{{d^{k - {\ell _1} + {r_1}}}}}{{d{y^{k - {\ell _1} + {r_1}}}}}{\left( {x + \frac{{\log ({y_2}/n)}}{{\log {y_2}}}} \right)^{{j}}}{\bigg|_{x = 0}}{\left( {y + \frac{{\log ({y_2}/n)}}{{\log {y_2}}}} \right)^{{i}}}{\bigg|_{y = 0}}\bigg] + O(T/L), \nonumber 
\end{align}
where we have used $\zeta(1+\alpha+\beta)=1/(\alpha+\beta)+O(1)$. We now sum over $i$ and $j$, e.g.
\[{P_{{\ell _1}}}\left( {x + \frac{{\log ({y_2}/n)}}{{\log {y_2}}}} \right) = \sum\limits_i {{b_{i,{\ell _1}}}{{\left( {x + \frac{{\log ({y_2}/n)}}{{\log {y_2}}}} \right)}^i}}, \]
thereby getting
\begin{align}
  I'_{220}(\alpha ,\beta ) &= \frac{{\widehat {w}(0)}}{{(\alpha  + \beta ){{\log }^2}{y_2}}}\frac{{{d^2}}}{{dxdy}}\bigg[y_2^{\alpha x + \beta y}\sum\limits_{{\ell _1} = 2}^K {\sum\limits_{{\ell _2} = 2}^K {\frac{1}{{{{\log }^{{\ell _1} + {\ell _2}}}{y_2}}}\sum\limits_{k = 0}^{\min ({\ell _1},{\ell _2})} {\binom{\ell_1}{k}{{({\ell _2})}_k}} } }  \nonumber \\
   &\quad \times \sum\limits_{n \leqslant {y_2}} {\sum\limits_{{r_1} = 0}^{{\ell _1} - k} {\sum\limits_{{r_2} = 0}^{{\ell _2} - k} {{{( - 1)}^{  {r_1} + {r_2}}}\binom{\ell_1 - k}{r_1}\binom{\ell_2 - k}{r_2}} } } \frac{{(\mathbf{1} * \Lambda_2 ^{ * k} * {\Lambda ^{ * {r_1} + {r_2}}})(n)}}{n} \nonumber \\
   &\quad \times \frac{{{d^{k - {\ell _2} + {r_2}}}}}{{d{x^{k - {\ell _2} + {r_2}}}}}\frac{{{d^{k - {\ell _1} + {r_1}}}}}{{d{y^{k - {\ell _1} + {r_1}}}}}{P_{{\ell _1}}}\left( {x + \frac{{\log ({y_2}/n)}}{{\log {y_2}}}} \right){\bigg|_{x = 0}}{P_{{\ell _2}}}\left( {y + \frac{{\log ({y_2}/n)}}{{\log {y_2}}}} \right){\bigg|_{y = 0}}\bigg] + O(T/L). \nonumber 
\end{align}
Lemma \ref{lemmaEM2lambdas} gives us
\begin{align}
  &\sum\limits_{n \leqslant {y_2}} {\frac{{(\mathbf{1} * {\Lambda_2 ^{ * k}} * \Lambda^{ * {r_1} + {r_2}})(n)}}{n}} {P_{{\ell _1}}}\left( {x + \frac{{\log ({y_2}/n)}}{{\log {y_2}}}} \right){P_{{\ell _2}}}\left( {y + \frac{{\log ({y_2}/n)}}{{\log {y_2}}}} \right) \nonumber \\
   &= \frac{{{2^{{r_1} + {r_2}}}{{\log }^{1 + 2k + r_1 + r_2}}{y_2}}}{(1+r_1+r_2+2k)!}\int_0^1 {{{(1 - u)}^{2k+r_1+r_2}}{P_{{\ell _1}}}(x + u){P_{{\ell _2}}}(y + u)du} + O({\log ^{2k + {r_1} + {r_2} }}{y_2}), \nonumber  
\end{align}
so that we we are left with
\begin{align}
  I'_{220}(\alpha ,\beta ) &= \frac{{\widehat {w}(0)}}{{(\alpha  + \beta)\log^{{2}} y_2 }}\frac{{{d^2}}}{{dxdy}}\bigg[y_2^{\alpha x + \beta y}\sum\limits_{{\ell _1} = 2}^K {\sum\limits_{{\ell _2} = 2}^K {\frac{1}{{{{\log }^{{\ell _1} + {\ell _2}}}{y_2}}}} } \sum\limits_{k = 0}^{\min ({\ell _1},{\ell _2})} {\binom{\ell_1}{k}{{({\ell _2})}_k}}  \nonumber \\
   &\quad \times \sum\limits_{{r_1} = 0}^{{\ell _1}} \sum\limits_{{r_2} = 0}^{{\ell _2}} {{{( - 1)}^{{r_1} + {r_2}}}} \binom{\ell_1-k}{r_1}\binom{\ell_2-k}{r_2} \frac{{{d^{k - {\ell _2} + {r_2}}}}}{{d{x^{k - {\ell _2} + {r_2}}}}}\frac{{{d^{k - {\ell _1} + {r_1}}}}}{{d{y^{k - {\ell _1} + {r_1}}}}} \nonumber \\
   &\quad \times \frac{{{2^{{r_1} + {r_2}}}{{\log }^{1+r_1+r_2+2k}}{y_2}}}{(1+r_1+r_2+2k)!}\int_0^1 {{{(1 - u)}^{2k + {r_1} + {r_2} }}{P_{{\ell _1}}}(x + u){P_{{\ell _2}}}(y + u)du} {\bigg|_{x = y = 0}}\bigg] + O(T/L). \nonumber 
\end{align}
Note that  $r_1\leq \ell_1-k$ and  $r_2\leq \ell_2-k$. Thus only the cases $r_1= \ell_1-k$ and  $r_2= \ell_2-k$ contribute to the main term. We therefore have
\begin{align}
  I'_{220}(\alpha ,\beta ) &= \frac{{\widehat {w}(0)}}{{(\alpha  + \beta)\log y_2 }}\frac{{{d^2}}}{{dxdy}}
  \bigg[y_2^{\alpha x + \beta y}\sum\limits_{{\ell _1} = 2}^K \sum\limits_{{\ell _2} = 2}^K  
  \sum\limits_{k = 0}^{\min ({\ell _1},{\ell _2})} {\binom{\ell_1}{k}{{({\ell _2})}_k}}  ( - 1)^{\ell_1+\ell_2-2k}   \nonumber \\
   &\quad \times \frac{2^{\ell_1+\ell_2-2k}}{(\ell_1+\ell_2)!}\int_0^1 {{{(1 - u)}^{\ell_1+\ell_2 }}{P_{{\ell _1}}}(x + u){P_{{\ell _2}}}(y + u)du} {\bigg|_{x = y = 0}}\bigg] + O(T/L). \nonumber 
\end{align}
Recall that
\[{I_{22}}(\alpha ,\beta ) = I{'_{22}}(\alpha ,\beta ) + {T^{ - \alpha  - \beta }}I{'_{22}}( - \beta , - \alpha ) + O(T/L),\]
and that
\[I{'_{22}}(\alpha ,\beta ) = I{'_{220}}(\alpha ,\beta ) + O({T^{1 - \varepsilon }}),\]
therefore
\begin{align*}
 {I_{22}}(\alpha ,\beta ) 
 &= I{'_{220}}(\alpha ,\beta ) + {T^{ - \alpha  - \beta }}I{'_{220}}( - \beta , - \alpha ) + O(T/L)\\
 &= (I{'_{220}}(\alpha ,\beta ) + I{'_{220}}(-\beta, -\alpha) )+ {(T^{ - \alpha  - \beta }-1)}I{'_{220}}( - \beta , - \alpha ) + O(T/L).
\end{align*}
We first take a look at the first term in the brackets
 \begin{align*}
  &\frac{d^2}{dxdy}\bigg[(y_2^{\alpha x + \beta y}-y_2^{-\beta x -\alpha y})\int_0^1 {{{(1 - u)}^{\ell_1+\ell_2 }}{P_{{\ell _1}}}(x + u){P_{{\ell _2}}}(y + u)du} {\bigg|_{x = y = 0}}\bigg]\\
  &=
  (\alpha+\beta)\log y_2 \bigg(\int_0^1 {{{(1 - u)}^{\ell_1+\ell_2 }}{P'_{{\ell _1}}}(u){P_{{\ell _2}}}(u)du} + \int_0^1 {{{(1 - u)}^{\ell_1+\ell_2 }}{P_{{\ell _1}}}(u){P'_{{\ell _2}}}(u)du}\bigg).
 \end{align*}
Since $P_{\ell _1}(0)=P_{\ell _2}(0)= 0$, we have also
 \begin{align*}
  &0 = {(1 - u)}^{\ell_1+\ell_2} P_{\ell _1}(u) P_{\ell _2}(u)\bigg|_{u = 0}^1 = \int_0^1 \bigg({(1 - u)}^{\ell_1+\ell_2} P_{\ell _1}(u) P_{\ell _2}(u)\bigg)'du.
   \end{align*}
This implies
 \begin{align*}
 &(\ell_1+\ell_2) \int_0^1 {(1 - u)}^{\ell_1+\ell_2-1} P_{\ell _1}(u) P_{\ell _2}(u) du\\
 =&
  \int_0^1 {{{(1 - u)}^{\ell_1+\ell_2 }}{P'_{{\ell _1}}}(u){P_{{\ell _2}}}(u)du} + \int_0^1 {{{(1 - u)}^{\ell_1+\ell_2 }}{P_{{\ell _1}}}(u){P'_{{\ell _2}}}(u)du}.
 \end{align*}
Combining these observations gives
\begin{align*}
 I{'_{220}}(\alpha ,\beta ) + I{'_{220}}(-\beta, -\alpha)
 =&
 \widehat {w}(0)
 \sum\limits_{{\ell _1} = 2}^K \sum\limits_{{\ell _2} = 2}^K  
  \sum\limits_{k = 0}^{\min ({\ell _1},{\ell _2})} ( - 1)^{\ell_1+\ell_2-2k}  \binom{\ell_1}{k} (\ell _2)_k \\
  &\times \frac{2^{\ell_1+\ell_2-2k}}{(\ell_1+\ell_2-1)!} \int_0^1 {(1 - u)}^{\ell_1+\ell_2-1} P_{\ell _1}(u) P_{\ell _2}(u) du.
\end{align*}
For the expression $(T^{ - \alpha  - \beta }-1)I{'_{22}}( - \beta , - \alpha )$, we again use \eqref{integraltrick} to get
\begin{align*}
  &\frac{{\widehat {w}(0)}}{\theta_2}
  \sum\limits_{{\ell _1} = 2}^K \sum\limits_{{\ell _2} = 2}^K  
  \sum\limits_{k = 0}^{\min ({\ell _1},{\ell _2})} {\binom{\ell_1}{k}{{({\ell _2})}_k}}  ( - 1)^{\ell_1+\ell_2-2k} \frac{2^{\ell_1+\ell_2-2k}}{(\ell_1+\ell_2)!}  \nonumber \\
   &\quad \times \frac{{{d^2}}}{{dxdy}}\bigg[y_2^{-\beta x -\alpha y}\int_0^1 \int_0^1T^{ - v(\alpha  + \beta )}{{{(1 - u)}^{\ell_1+\ell_2 }}{P_{{\ell _1}}}(x + u){P_{{\ell _2}}}(y + u)dudv} {\bigg|_{x = y = 0}}\bigg] + O(T/L).  \nonumber
 \end{align*}
By using similar arguments for the holomorphy of the error terms as in the Section 3.1, we end the proof of Lemma \ref{lemmac22}.
\section{Acknowledgments}
The authors are grateful to Arindam Roy and Alexandru Zaharescu for fruitful discussions. The first author wishes to acknowledge partial support from SNF grant PP00P2 138906.

\end{document}